\documentclass[12pt]{amsart}

\usepackage{amsmath,amssymb,amsthm,graphicx,color,amscd}
\usepackage[usenames,dvipsnames]{xcolor}
\usepackage{enumitem}

\usepackage[small,nohug,heads=vee]{diagrams}

\usepackage{hyperref}

\numberwithin{equation}{section}
\hypersetup{bookmarksdepth=3}

\theoremstyle{plain}
\newtheorem{theorem}[equation]{Theorem}

\newtheorem{conjecture}[equation]{Conjecture}
\newtheorem{regularity_conjecture}[equation]{Regularity Conjecture}
\newtheorem{corollary}[equation]{Corollary}

\newtheorem{proposition}[equation]{Proposition}
\newtheorem{lemma}[equation]{Lemma}

\theoremstyle{definition}
\newtheorem{definition}[equation]{Definition}

\theoremstyle{remark}
\newtheorem{remark}[equation]{Remark}

\newcommand{\al}{\alpha}

\newcommand{\aut}{\operatorname{Aut}}

\newcommand{\be}{\beta}
\newcommand{\ben}{\begin{enumerate}}
\newcommand{\bit}{\begin{itemize}}
\newcommand{\C}{\mathbb{C}}

\newcommand{\com}{\operatorname{com}}
\newcommand{\D}{\partial}
\newcommand{\de}{\delta}

\newcommand{\diam}{\operatorname{diam}}

\newcommand{\een}{\end{enumerate}}
\newcommand{\eit}{\end{itemize}}

\newcommand{\eps}{\varepsilon}
\newcommand{\esssup}{\operatorname{esssup}}

\newcommand{\fg}{\mathfrak{g}}
\newcommand{\fh}{\mathfrak{h}}

\newcommand{\ga}{\gamma}

\newcommand{\h}{\mathcal{H}}
\renewcommand{\H}{\mathbb{H}}

\newcommand{\I}{\mathcal{I}}
\newcommand{\id}{\operatorname{id}}

\newcommand{\Int}{\operatorname{Int}}

\renewcommand{\Im}{\operatorname{Im}}

\newcommand{\la}{\lambda}
\newcommand{\La}{\Lambda}

\newcommand{\loc}{\operatorname{loc}}
\newcommand{\lra}{\longrightarrow}

\newcommand{\ol}{\overline}
\newcommand{\om}{\omega}
\newcommand{\Om}{\Omega}

\newcommand{\perm}{\operatorname{Perm}}

\newcommand{\R}{\mathbb{R}}
\newcommand{\ra}{\rightarrow}

\definecolor{gray}{gray}{0.7}

\newcommand{\restr}{\mbox{\Large \(|\)\normalsize}}

\newcommand{\si}{\sigma}

\newcommand{\Span}{\operatorname{span}}
\newcommand{\spt}{\operatorname{spt}}

\renewcommand{\th}{\theta}

\newcommand{\ul}{\underline}
\newcommand{\vol}{\operatorname{vol}}

\newcommand{\we}{\wedge}

\newcommand{\wt}{\operatorname{wt}}

\newcommand{\ze}{\zeta}

\def\Xint#1{\mathchoice
{\XXint\displaystyle\textstyle{#1}}%
{\XXint\textstyle\scriptstyle{#1}}%
{\XXint\scriptstyle\scriptscriptstyle{#1}}%
{\XXint\scriptscriptstyle\scriptscriptstyle{#1}}%
\!\int}
\def\XXint#1#2#3{{\setbox0=\hbox{$#1{#2#3}{\int}$ }
\vcenter{\hbox{$#2#3$ }}\kern-.6\wd0}}

\def\av{\Xint-}

\begin{document}

\title[Rigidity of mappings between Carnot groups]{Pansu pullback and rigidity of mappings between Carnot groups}

\author{Bruce Kleiner}
\thanks{BK was supported by NSF grants DMS-1405899, DMS-1711556,  DMS-2005553 and a Simons Collaboration grant.}
\email{bkleiner@cims.nyu.edu}
\address{Courant Institute of Mathematical Science, New York University, 251 Mercer Street, New York, NY 10012}
\author{Stefan M\"uller}
\thanks{SM has been supported by the Deutsche Forschungsgemeinschaft (DFG, German Research Foundation) through
the Hausdorff Center for Mathematics (GZ EXC 59 and 2047/1, Projekt-ID 390685813) and the 
collaborative research centre  {\em The mathematics of emerging effects} (CRC 1060, Projekt-ID 211504053).  This work was initiated during a sabbatical of SM at the Courant Institute and SM would like to thank  R.V. Kohn and the Courant Institute
members and staff for 
their  hospitality and a very inspiring atmosphere.}
\email{stefan.mueller@hcm.uni-bonn.de}
\address{Hausdorff Center for Mathematics, Universit\"at Bonn, Endenicher Allee 60, 53115 Bonn}
\author{Xiangdong Xie}
\thanks{XX has been supported by Simons Foundation grant \#315130.}
\email{xiex@bgsu.edu}
\address{Dept. of Mathematics and Statistics, Bowling Green State University, Bowling Green, OH 43403}

\maketitle

\begin{abstract}
 This is the first in a series of papers \cite{KMX2,kmx_approximation_low_p,kmx_rumin,kmx_iwasawa} on geometric mapping theory in Carnot groups -- and more generally equiregular manifolds -- in which we  prove a number of new structural results for Sobolev (in particular quasisymmetric) mappings, establishing (partial) rigidity or (partial) regularity theorems, depending on the context.

\end{abstract}

\tableofcontents

\section{Introduction}

\subsection*{Background and overview}

\mbox{}
Quasiconformal mappings between domains in $\R^n$ have been heavily studied since the 1930s.  They initially arose in Teichmuller theory, and over time found applications in a number of areas of mathematics, including Kleinian groups, complex dynamics, PDE, rigidity of lattices, and geometric group theory. Quasiconformal mappings in non-Euclidean settings first appeared around 1970 in the work of Mostow on rigidity of lattices in Lie groups, and the topic gained momentum in the 80s and 90s due to the convergence of developments from several directions, including Gromov's notion of hyperbolicity, fundamental work of Pansu on quasi-isometries of symmetric spaces, advances in the analytical theory of quasiconformal homeomorphisms, and the explosion of interest in analysis on metric spaces (see for instance \cite{gromov_hyperbolic_groups,pansu,iwaniec_martin_quasiregular_even_dimensions,margulis_mostow_differential_quasiconformal_mapping,heinonen_koskela,astala_icm,bonk_icm,kleiner_icm,heinonen_nonsmooth}).  A new phenomenon in the non-Euclidean setting is rigidity \cite{pansu}: in some situations quasiconformal mappings turn out to be very restricted, in sharp contrast to mappings in $\R^n$, which are flexible and come in infinite dimensional families.    A heuristic explanation for rigidity is that in the case of smooth mappings quasiconformality is equivalent to a first order system which is formally overdetermined except in very special cases.  However, the limited regularity of quasiconformal homeomorphisms has posed a major obstacle to converting this heuristic into  theorems:  since quasiconformal homeomorphisms are not a priori smooth, the standard technique for exploiting overdetermined conditions -- repeated differentiation -- is not applicable.  For this reason, rigidity is only known in some special cases where the regularity problem has been successfully overcome or circumvented    \cite{pansu,capogna,capogna_cowling,xie_filiform}.  
Regularity issues aside, the general picture remains unclear even for smooth quasiconformal homeomorphisms,  because the precise implications of the condition depend in a subtle way on the structure of the Carnot group $G$; see below for further discussion.

In a series of papers \cite{KMX2,kmx_approximation_low_p,kmx_rumin,kmsx_infinitesimally_split_globally_split,kmsx_counterexample,kmx_iwasawa}  we establish a number of rigidity and regularity theorems for quasiconformal homeomorphisms, and more generally, Sobolev mappings.  A key technical result is a generalization to Carnot groups of the fact that in Euclidean space, pullback of differential forms by Sobolev mappings commutes with exterior differentiation, under suitable assumptions on the Sobolev exponent \cite{reshetnyak_space_mappings_bounded_distortion}.  We consider Sobolev (in particular quasiconformal) mappings  $f:G\supset U\ra U'\subset G'$ between open subsets of Carnot groups, and define the Pansu pullback $f_P^*\om$ of a smooth differential form $\om\in \Om^*(U')$ using the (approximate) Pansu differential.  Although Pansu pullback does not commute with exterior differentiation in general,   even for smooth Sobolev mappings (see Lemma~\ref{lem_pullback_exterior_derivative}), it turns out that a partial analog does hold: for some closed forms $\om$, certain components of the distributional exterior derivative of $f_P^*\om$ are zero.   This result provides new constraints on mappings which can be exploited in different ways depending on the context -- to exclude oscillatory behavior of the Pansu differential, to prove regularity, or to show that auxiliary geometric objects satisfy a PDE.  We note that Dairbekov and Vodopyanov, motivated by Reshetnyak's work and applications to quasiregular mappings, considered Pansu pullback   and showed that it  commutes with exterior derivative in special cases  \cite{dairbekov_morphism_property_bounded_distortion,vodopyanov_bounded_distortion,vodopyanov_foundations}.
 
Although the initial motivation for this work arose from geometric mapping theory and geometric group theory, in spirit the phenomena and methods have much in common with the analytical side of geometric mapping theory, and with the literature on rigidity and oscillatory solutions to PDEs, see for instance 
\cite{reshetnyak_space_mappings_bounded_distortion,iwaniec_martin_quasiregular_even_dimensions,vodopyanov_foundations,nash54, tartar79, murat81, gromov_pdr, scheffer93, dacorogna_marcellini99, muller99, muller_sverak03, delellis_szekelyhidi09, delellis_szekelyhidi16, isett18, buckmaster_vicol19}.

\bigskip
\subsection*{Statement of results}~
Here we will cover the results from this paper and \cite{KMX2}; 
we refer the reader to \cite{kmx_approximation_low_p,kmx_rumin,kmsx_infinitesimally_split_globally_split} for further results.  To avoid technicalities, in this paper we restrict our attention to Sobolev exponents which imply continuity; this includes in particular quasiconformal mappings.  See \cite{kmx_approximation_low_p,kmsx_infinitesimally_split_globally_split} for rigidity, regularity, and flexibility results concerning Sobolev mappings with lower integrability exponents.   Many of the results for mappings between Carnot groups discussed in the introduction hold in the more general setting of equiregular manifolds; see Appendix~\ref{sec_equiregular_manifolds}.

We begin by setting notation and briefly recalling a few facts about Carnot groups; see Section~\ref{sec_prelim} for more detail.  

Let $G$ be a Carnot group with Lie algebra $\fg$, grading $\fg=\oplus_{j\geq 1}^sV_j$, and  dilation group $\{\de_r:G\ra G\}_{r\in (0,\infty)}$.  Without explicit mention, in what follows all Carnot groups will be equipped with Haar measure and a Carnot-Caratheodory metric denoted generically by $d_{CC}$.  We recall that if $f:G\supset U\ra G'$ is a Sobolev mapping between Carnot groups, where $U$ is open, then    $f$ has a well-defined  Pansu differential $D_Pf(x):G\ra G'$ for a.e. $x\in U$ provided $f\in W^{1,p}_{\loc}$ for some $p$ strictly larger than the Hausdorff (homogeneous) dimension of $G$.  The Pansu differential $D_Pf(x):G\ra G'$ is a graded group homomorphism, which we often conflate with the associated homomorphism of graded Lie algebras $D_Pf(x):\fg\ra \fg'$.  If $f:G\supset U\ra U'\subset G$ is a quasiconformal homeomorphism, then $f\in W^{1,p}_{\loc}$ for some $p$ strictly larger than the Hausdorff (homogeneous) dimension of $G$, and $D_Pf(x)$ is an isomorphism for a.e. $x\in U$.
See Theorem \ref{thm_properties_qc_homeomorphisms}.

We first discuss mappings and product structure.  We remark that the problem of determining to what extent a mapping $X_1\times X_2\ra Y_1\times Y_2$ must respect the product structure has arisen in various places in mathematics  \cite{scottish_book,whitehead_homotopy_type,fox_problem_ulam,fournier_ulam,jonsson_direct_decompositions}; our original motivation was the product rigidity theorem in geometric group theory \cite{kkl_qis_and_de_rham}.

We observe that bilipschitz homeomorphisms need not respect product structure in Carnot groups: if $G$ is a Carnot group then for any nonconstant Lipschitz map $\al:G\ra \R^n$ the  shear $(x,y)\mapsto (x,y+\al(x))$ defines a bilipschitz homeomorphism $G\times\R^n\ra G\times \R^n$ which does not respect the product structure.  Another type of example comes from products of the form $G_1\times G_2$, where $G_1\simeq K$, $G_2\simeq K\times K$ for some Carnot group $K$: the map $(x_1,(x_2,x_3))\mapsto (x_3,(x_1,x_2))$ does not respect the given product structure.  Our first result asserts that nondegenerate Sobolev mappings between products must respect the product structure once these two types of examples are excluded.  Let $\{G_i\}_{1\leq i\leq n}$, $\{G_j'\}_{1\leq j\leq n'}$ be  collections of Carnot groups, where each $G_i$, $G_j'$ is nonabelian and does not admit a  nontrivial decomposition as a product of Carnot groups.  
Let $G:=\prod_{i}G_i$,  $G':=\prod_{j}G'_j$.  

\begin{theorem}
\label{thm_main_product_intro}
Let $\nu$ denote the homogeneous dimension of $G$.  Suppose for some $p>\nu$ that   $f:G\supset U\ra G'$ is   a   $W^{1,p}_{\loc}$-mapping, $U=\prod_iU_i$ is a product of connected open sets $U_i\subset G_i$, and the Pansu differential $D_Pf(x)$ is an isomorphism for a.e. $x\in U$.  Then  $f$ is a product of mappings, i.e. $n=n'$,  and for  some permutation $\si:\{1,\ldots,n\}\ra \{1,\ldots,n\}$ there are mappings  $\{f_{\si(i)}: U_i \ra G_{\si(i)}'\}_{1\leq i\leq n}$ such that 
\begin{equation}
\label{eqn_f_product}
f(x_1,\ldots,x_n)=(f_1(x_{\si^{-1}(1)}),\ldots,f_n(x_{\si^{-1}(n)}))
\end{equation}
for every $(x_1,\ldots,x_n)\in U  $.   In particular, any quasiconformal homeomorphism $G\ra G'$ is a product of homeomorphisms.
\end{theorem}

\begin{remark}
It is a very intriguing problem to understand the dependence of the rigidity phenomenon in Theorem~\ref{thm_main_product_intro} on the Sobolev exponent.  In \cite{kmx_approximation_low_p} we show that rigidity persists if $p\geq \max \nu_i - 1$, where $\nu_i$ is the homogeneous dimension of $G_i$; in fact, in some cases even $p \ge 2$ is sufficient.
 In \cite{kmsx_infinitesimally_split_globally_split} we study a related rigidity problem in Euclidean space, determining the threshold exponent at which rigidity transitions to flexibility.
\end{remark}

Theorem~\ref{thm_main_product_intro} has an application to negatively curved homogeneous spaces.  If  $G$, $G'$ are as in Theorem~\ref{thm_main_product_intro} for $n\geq 2$, and $X$, $X'$ are the negatively curved homogeneous spaces arising from  the construction of Heintze \cite{heintze}, then every quasi-isometry $X\ra X'$ is a rough isometry up to rescaling the metric on $G'$; see \cite{Xie_Pacific2013} for details.  

By means of a blow-down argument, we use the product rigidity theorem above to  see that quasi-isometries between products must asymptotically preserve product structure.  

Let $\{G_i\}_{1\leq i\leq n}$, $\{G_j'\}_{1\leq j\leq n'}$ be as above, and let   $\{\hat G_i\}_{1\leq i\leq n}$, $\{\hat G_j'\}_{1\leq j\leq n'}$   be  collections where:
\bit
\item Each $\hat G_i$, $\hat G'_j$ is either a simply-connected nilpotent Lie group with a left invariant Riemannian metric, or a finitely generated nilpotent group equipped with a word metric.
\item  For every $1\leq i\leq n$, $1\leq j\leq n'$ the asymptotic cone of  $\hat G_i$, $\hat G_j'$ is bilipschitz homeomorphic to $G_i$, $G_j'$, respectively \cite{pansu_croisssance}. 
\eit 
Set $\hat G:=\prod_i\hat G_i$, $\hat G'=\prod_jG_j'$, and equip $\hat G$, $\hat G'$ with the $\ell^2$-distance functions, 
$$
d^2_{\hat G}(x,x'):={\sum_id_{G_i}^2(\pi_i(x),\pi_i(x'))}\,,\quad
d^2_{\hat G'}(x,x'):={\sum_jd_{G_j'}^2(\pi_j(x),\pi_j(x'))}\,.
$$
If $1\leq i\leq n$, we say that two points $x=(x_1,\ldots,x_n),\,x'=(x_1',\ldots,x_n')\in G$ form an {\bf $i$-pair} if $x_j=x_j'$ for every $j\neq i$.  The next result says that an $i$-pair $x,x'\in \hat G$ with $d(x,x')$ large is nearly mapped to a $j$-pair, where $j$ depends only on $i$.  
\begin{theorem}
\label{thm_qis_between_products_intro}
For every $L\geq 1$, $A<\infty$, there is a function $\eps=\eps_{L,A}:[0,\infty)\ra (0,1]$ with $\lim_{R\ra\infty}\eps(R)=0$ with the following property.

For every $(L,A)$-quasi-isometry $\Phi:\hat G\ra \hat G'$, there is a bijection $\si:\{1,\ldots,n\}\ra \{1,\ldots,n'\}$, such that for every $R\in [0,\infty)$, $1\leq i\leq n$, and every $i$-pair  $x,x'\in \hat G$ with $d(x,x')\geq R$, we have 
\begin{equation}
\label{eqn_asymptotically_product}
\frac{d(\pi_j(\Phi(x)),\pi_j(\Phi(x')))}{d(\Phi(x),\Phi(x'))}\in
\begin{cases}
(1-\eps,1]\quad\text{if}\quad j=\si(i)\\
[0,\eps)\quad\text{if}\quad j\neq \si(i)\,.
\end{cases}
\end{equation}
\end{theorem} 
\noindent
One may compare Theorem~\ref{thm_qis_between_products_intro} with the rigidity theorem for quasi-isometries 
between products of Gromov hyperbolic spaces \cite{kkl_qis_and_de_rham}; in that case the quasi-isometry is at finite sup distance from a product of quasi-isometries, modulo reindexing factors.  This stronger assertion is false in the setting of Theorem~\ref{thm_qis_between_products_intro}, because of central shears, see \cite{xie_some_examples_qis_nilpotent}.  Nonetheless, Theorem~\ref{thm_qis_between_products_intro}   
    is still strong enough to imply that the factors are quasi-isometric: 
\begin{corollary}
\label{cor_factors_are_qi_intro}
If $\hat G$, $\hat G'$, and $\Phi$ are as in Theorem~\ref{thm_qis_between_products_intro}, then modulo reindexing, $\hat G_i$ is quasi-isometric to $\hat G_i'$.
\end{corollary}
\noindent
We recall that the quasi-isometry classification of nilpotent groups has remained a major open problem in geometric group theory since Pansu's work \cite{pansu,shalom}; Corollary~\ref{cor_factors_are_qi_intro} shows that this reduces to the case of groups which are indecomposable (in an appropriate sense).

Our next result is a regularity theorem for nondegenerate Sobolev maps between certain complexified Carnot groups.  For simplicity we only state the result for complex Heisenberg groups here, and refer the reader to Section~\ref{sec_complex_heisenberg_setup} for the general case.
\begin{theorem}
\label{thm_qc_complex_heisenberg_holo_antiholo_intro}
Let $\H_n^\C$ denote the complexification of the $n^{th}$ Heisenberg group $\H_n$, and let $U\subset \H_n^\C$ be a connected open set.  If  $f:U\ra \H_n^\C$ belongs to $W^{1,p}_{\loc}$ for some $p >4n+2$, and the Pansu differential $D_Pf(x)$ is an isomorphism for a.e. $x\in U$, then $f$ is holomorphic or antiholomorphic.   
\end{theorem}

\begin{corollary}
\label{cor_qc_complex_heisenberg_holo_antiholo}
Any quasiconformal homeomorphism $\H_n^\C\supset U\ra U'\subset \H_n^\C$ between connected open subsets is either holomorphic or antiholomorphic.
\end{corollary}
\noindent
In the case of $C^2$ diffeomorphisms Corollary~\ref{cor_qc_complex_heisenberg_holo_antiholo} is due to Reimann-Ricci \cite{reimann_ricci}.     We also show that global quasiconformal homeomorphisms are rigid:   

\begin{theorem}
\label{thm_global_qc_ch_affine_intro}
The group of quasiconformal homeomorphisms $\H_n^\C\ra\H_n^\C$ is generated by left translations, complex graded automorphisms, and complex conjugation.
\end{theorem}
\noindent
In contrast to global quasiconformal homeomorphisms, locally defined quasiconformal homeomorphisms $\H_n^\C\supset U\ra U'\subset \H_n^\C$ are quite flexible, and come in infinite dimensional families, see Section~\ref{sec_complex_heisenberg_setup}.  

We now shift our attention to quasiconformal homeomorphisms $G\supset U\ra U'\subset G$ for a general Carnot group $G$.  For a smooth diffeomorphism, being locally quasiconformal is  equivalent to being a contact diffeomorphism,    i.e. preservation of the subbundle of the tangent bundle $TG$ defined by the first layer $V_1\subset \fg$.  A parameter count indicates that the  contact condition is formally overdetermined unless $G$ is isomorphic to the Engel group or to $\R^k\times \H_\ell$ for some $k,\ell$, and hence one expects some form of rigidity.  However, the analytical character of the condition is quite different for different groups: 
\bit
\item When $G=\H\times\H$ contact diffeomorphisms must be products (locally) but otherwise are quite flexible \cite{cowling_reimann_three_examples}.
\item When $G=\H_n^\C$ the contact condition is locally flexible, but still ``hypoelliptic'', i.e. contact diffeomorphisms are holomorphic or antiholomorphic.
\item When $G$ is an $H$-type group with center of dimension at least $3$, (e.g. one of the Carnot groups studied by Pansu) or a free Carnot  group of step $s\geq 3$,  then the smooth contact embeddings $G\supset U\ra G$ form a finite dimensional family when $U$ is a connected open subset \cite{pansu,reimann_h_type,warhurst_tanaka_prolongation_free}.
\eit
Following Ottazzi-Warhurst  \cite{{ottazzi_warhurst}}, for $1\leq k\leq \infty$ we say that a Carnot group $G$ is {\bf $C^k$-rigid} if the space of $C^k$ contact embeddings $G\supset U\ra G$ forms a finite dimensional family for every connected open subset $U\subset G$. Building on the theory of overdetermined systems \cite{guillemin_quillen_sternberg,spencer}, Ottazzi-Warhurst \cite{{ottazzi_warhurst}} gave several different characterizations of $C^\infty$-rigid groups, and proved the following regularity theorem:
\begin{theorem}[\cite{ottazzi_warhurst}]
\label{thm_ow_c2_smooth}
If $G$ is a $C^\infty$-rigid Carnot group, then any $C^2$-contact diffeomorphism $G\supset U\ra U'\subset G$ is $C^\infty$-smooth\footnote{Recently  Jonas Lelmi improved Theorem~\ref{thm_ow_c2_smooth}, replacing the $C^2$ regularity assumption with $C^1$ (or even Euclidean bilipschitz); the same result was shown 
by Alex Austin for  the $(2,3,5)$ distribution \cite{lelmi,austin_235}.  It is easy to see that smooth contact diffeomorphisms are actually real analytic.}.  
\end{theorem}
\noindent
In view of the theorem, a group is $C^k$-rigid for $k\geq 2$ if and only if it is  $C^\infty$-rigid, so we will call such groups {\bf rigid}.  This leads naturally to the following (cf. \cite[p.2]{ottazzi_warhurst}):
\begin{regularity_conjecture}\label{conj_regularity_conjecture}
If $G$ is a rigid Carnot group, then any quasiconformal homeomorphism $G\supset U\ra U'\subset G$ is $C^\infty$.  
\end{regularity_conjecture}
\noindent
The conjecture is known for groups whose graded automorphisms act conformally on the first layer, by subelliptic regularity \cite{capogna,capogna_cowling}; it follows from Theorem~\ref{thm_main_product_intro} that the Regularity Conjecture holds for a product of Carnot groups $\prod_iG_i$ if it holds for all of the factors $\{G_i\}$.  We will present further results on rigid groups elsewhere \cite{kmx_iwasawa}, and focus here on the case of nonrigid groups.  Our main result for such groups is that one always has partial rigidity,  apart from some exceptional cases:
\begin{theorem}[See Section 4 of \cite{KMX2}  for definitions]
\label{thm_nonrigid_intro}
If $G$ is a nonrigid Carnot group with homogeneous dimension $\nu$, then one of the following holds:
\ben
\item $G$ is isomorphic to $\R^n$ or to a real or complex Heisenberg group $\H_n$, $\H_n^\C$ for some $n\geq 1$.
\item There is a closed subgroup $\{e\}\subsetneq H\subsetneq G$, a constant $K$, and a finite    set   $A$ of graded automorphisms of $G$ with the following properties:    
\bit
\item For every $p>\nu$, $x\in G$, $r\in(0,\infty)$, and every $W^{1,p}_{\loc}$-mapping $$f:G\supset B(x,r)\ra G$$ such that the sign of $\det(D_Pf)$ is constant almost everywhere, then for some $\Phi\in A$ the restriction of the composition $\Phi\circ f$ to the subball $B(x,\frac{r}{K})$ preserves the coset foliation of $H$. In particular, the conclusion holds for quasiconformal homeomorphisms.
\item The Lie algebra of $H$ is generated by a linear subspace $\{0\}\subsetneq W\subsetneq V_1$ with $[W,V_j]=\{0\}$ for all $j\geq 2$.
\eit
\een
\end{theorem}
Thus, apart from the exceptional cases in (1), quasiconformal homeomorphisms  preserve a foliation, up to post-composition with a graded automorphism.   This suggests that for the (nonexceptional) nonrigid cases there may be a more detailed description of quasiconformal homeomorphisms along the lines of \cite{xie_filiform} for model filiform groups.  To our knowledge such a description is not known even for smooth contact diffeomorphisms, either locally or globally.

Another aspect of quasiconformal regularity/flexibility has to do with the Sobolev exponent, i.e. higher integrability of the derivative.  Quasiconformal homeomorphisms $f:G\supset U\ra U'\subset G$ are always in $W^{1,p}_{\loc}$ for some $p$ strictly larger than the homogeneous dimension of $G$, where $p$  depends on $G$ and the quasiconformal distortion of $f$ \cite{heinonen_koskela}.  However, except for $\R^n$ and the Heisenberg groups $\H_n$ (see for example \cite{balogh_non_bilipschitz}), all known examples of quasiconformal homeomorphisms are in $W^{1,\infty}_{\loc}$, i.e. are locally bilipschitz.  This motivated the following:
\begin{conjecture}[Xie]
\label{conj_qs_rigidity}
If $G$ is a Carnot group other than $\R^n$ or $\H_n$ for some $n$, then every quasiconformal homeomorphism $f:G\supset U\ra U'\subset G$ is locally bilipschitz; moreover, if $U=G$ then $f$ is bilipschitz.
\end{conjecture}
Using Theorem~\ref{thm_nonrigid_intro} and a variation on \cite{Shan_Xie, Xie_Pacific2013,LeDonne_Xie} we show:
\begin{theorem}
Conjecture~\ref{conj_qs_rigidity} holds for nonrigid Carnot groups.  
\end{theorem}
\noindent
See Theorem 4.1 in \cite{KMX2}  for a more precise statement.
Note that for rigid groups Conjecture~\ref{conj_qs_rigidity} would follow from the Regularity Conjecture and \cite{cowling_ottazzi}.   

Our next result, which is a key ingredient in the rigidity theorems above, concerns the interaction between Sobolev mappings and differential forms.  We recall that pullback of smooth $k$-forms by a $W^{1,k+1}_{\loc}$-mapping $f:\R^n\supset U\ra \R^m$  commutes with the exterior derivative, i.e.
\begin{equation}
\label{eqn_sobolev_pullback_theorem_rn}
df^*\om=f^*d\om
\end{equation}   
for all $\om\in \Om^k(\R^m)$; here the exterior derivative on the left hand side is interpreted distributionally.  To our knowledge this was first observed by Reshetnyak, who used it in his approach to quasiconformal and quasiregular mappings \cite{reshetnyak_space_mappings_bounded_distortion,reshetnyak67-distorsion}; it has many other applications to geometric mapping theory in Euclidean space.  Now let $f:G\supset U\ra G'$ be a $W^{1,p}_{\loc}$ mapping where $G$, $G'$ are Carnot groups and $U$ is an open subset.   Since $f$ need not be a Sobolev mapping with respect to the Euclidean metrics  (\ref{eqn_sobolev_pullback_theorem_rn}) is not applicable.  On the other hand we may define the  {\bf Pansu pullback} of a smooth $k$-form  $\om\in \Om^k(U')$ to be  the $k$-form with measurable coefficients $f_P^*\om$ given by 
$$
f_P^*\om(p)=(D_Pf(p))^*\om(f(p))\,;
$$
here we are viewing the Pansu differential as a graded homomorphism $D_Pf(p):\fg\ra \fg'$, and identifying tangent spaces of $G$, $G'$ with their respective Lie algebras via left translation.
It turns out that unlike ordinary pullback, Pansu pullback does not commute with the exterior derivative, even for smooth contact diffeomorphisms (see Lemma~\ref{lem_pullback_exterior_derivative}).   Nonetheless, there is a partial generalization of  (\ref{eqn_sobolev_pullback_theorem_rn}) to Carnot groups. To state this, we will need the notion of the weight of a differential form; the weight $\wt(\al)$ of a differential form $\al$ is defined using its decomposition with respect to the diagonalizable action of the Carnot scaling on $\La^*\fg$, see Section~\ref{se:weights}.

\begin{theorem}[Pullback theorem]  \label{th:pull_back_intro} 
Let $G$, $G'$ be Carnot groups where $G$ has dimension $N$ and homogeneous dimension $\nu$.
Let $f:G\supset U\ra U'\subset G'$ be a $W^{1,p}_{\loc}$-mapping between open subsets, for some $p>\nu$.   Suppose $\om\in \Om^k(U')$ and  $\eta\in \Om^{N-k-1}_c(U)$ are smooth forms such that 

$$
\wt(\om)+\wt(d\eta)\leq -\nu\,,\quad \wt(d\om)+\wt(\eta)\leq -\nu\,,
$$    
then 
\begin{equation}
\label{eqn_pullback_wedge_deta}
\int_U(f_P^*\om)\we\eta+(-1)^k\int_Uf_P^*\om\we d\eta=0\,.
\end{equation}
In particular, if $\om\in\Om^k(U')$ and $\zeta\in \Om^{N-k-1}(U)$ are closed and $\wt(\om)+\wt(\zeta)\leq -\nu+1$, then the distributional exterior derivative $df_P^*\om$ satisfies
$$
\zeta\wedge df_P^*\om=0\,.
$$
\end{theorem}

Theorem~\ref{th:pull_back_intro} is a consequence of an approximation theorem (see Theorem~\ref{thm_approximation_intro} below), which is the key technical result in this paper.   Theorem~\ref{thm_approximation_intro} has a number of other applications, which we will discuss elsewhere.  Dairbekov and Vodopyanov were the first (to our knowledge) to recognize the potential importance of a generalization of (\ref{eqn_sobolev_pullback_theorem_rn}) to mappings between Carnot groups, in particular for quasiregular mappings \cite{dairbekov_morphism_property_bounded_distortion,vodopyanov_bounded_distortion,vodopyanov_foundations}. In \cite{vodopyanov_foundations} it was shown that (\ref{eqn_sobolev_pullback_theorem_rn}) holds in the case  of   step 2 Carnot groups, for certain forms $\om$ of codegree $1$; this result also follows immediately from Theorem~\ref{thm_approximation_intro} for $W^{1,p}$-mappings with $p>\nu$.

\begin{remark}
\label{rem_stronger_version_exists}
In \cite{kmx_approximation_low_p} we prove a stronger version of Theorems~\ref{th:pull_back_intro} and \ref{thm_approximation_intro}, by relaxing the condition on the Sobolev exponent $p$, allowing $p\leq  \nu$; in particular it implies stronger versions of Theorems~\ref{thm_main_product_intro}, \ref{thm_qc_complex_heisenberg_holo_antiholo_intro}, and \ref{thm_nonrigid_intro}.  Although the basic outline of the proof is similar to that of Theorem~\ref{thm_approximation_intro},  the fact that $p\leq \nu$ creates several complications: a Sobolev mapping $f\in W^{1,p}_{\loc}(U,G')$ as in the theorem need not be either Pansu differentiable almost everywhere or continuous; in particular, the argument cannot be localized in the target.   
\end{remark}

\bigskip
\subsection*{Discussion of proofs}
We begin with Theorem~\ref{thm_main_product_intro}, and for simplicity we take  $U=G=G'=\H\times\H$; see Section~\ref{sec_rigidity_products} for the detailed proof.

Consider a $W^{1,p}_{\loc}$-mapping  $f:\H\times\H\ra \H\times \H$   for some $p>8$, such that the Pansu differential $D_Pf(x)$ is an isomorphism for a.e. $x\in \H\times \H$.  Since the Pansu differential $D_Pf(x):\H\times\H\ra \H\times \H$ is a graded isomorphism of Carnot groups, it is easy to see that $D_Pf(x)$ must respect the direct sum decomposition $\fg=\fh\oplus\fh$, i.e. there is a permutation  $\si_x:\{1,2\}\ra\{1,2\}$ such that 
\begin{equation}
\label{eqn_derivative_splits}
D_Pf(x)(x_1,x_2)=(h_1(x_{\si_x^{-1}(1)}),h_2(x_{\si_x^{-1}(2)}))
\end{equation} 
for some automorphisms $h_1,h_2:\fh\ra\fh$.  If the permutation $\si_x$ is independent of $x$, then one may readily integrate (\ref{eqn_derivative_splits}) to deduce that $f$ itself is a product as in (\ref{eqn_f_product}).  If $D_Pf(x)$ varied continuously with $x$, it would follow immediately that the permutation $\si_p$ is constant since it takes values in a discrete set; however $D_Pf(x)$ varies only measurably with $x$.  To appreciate the delicacy of this issue, consider the Lipschitz folding map $F: \R^2\ra\R^2$ defined by 
$$
F(x)=
\begin{cases}
x\quad \text{if}\quad x\in H_+\\
r_\ell(x)\quad\text{if}\quad x\in H_-\,,
\end{cases}
$$
where $H_\pm\subset \R^2$ are the two halfplanes bounded by the line $\ell:=\{(x,x)|x\in \R\}$, and $r_\ell:\R^2\ra\R^2$ is reflection across the line $\ell$.  The derivative $DF(x):\R^2\ra \R^2$ preserves the direct sum decomposition $\R^2=\R\oplus\R$ for a.e. $x\in \R^2$, but $F$ is not a product mapping.      We are able to exclude this kind of behavior in the permutation $\si_x$ from (\ref{eqn_derivative_splits}) using the Pullback Theorem and a short argument with differential forms.\footnote{Our first  proof of the product rigidity theorem used currents and the pullback theorem; see Subsection~\ref{subsec_currents_proof} for a brief sketch.  We thank Mario Bonk for  the simpler dual argument using forms, which he suggested  to us after hearing the original proof.} 

Although the details are different, the starting point of the proof  of    Theorem~\ref{thm_qc_complex_heisenberg_holo_antiholo_intro} 
   is also the exclusion of oscillatory behavior.  If $f:U\ra \H_n^\C$ is as in the theorem, then the Pansu differential $D_Pf(x):\fh_n^\C\ra \fh_n^\C$ is a graded automorphism for a.e. $x\in U$ by assumption, and is therefore either  
$\C$-linear or $\C$-antilinear  (see \cite{reimann_ricci}  or  Lemma 5.6 in \cite{KMX2}).  Using the pullback theorem one can show that the map $f$ has to make a global choice between these two possibilities.  The remainder of the proof is based on more standard PDE techniques, see Section~\ref{sec_complex_heisenberg_setup}.

We refer the reader  to \cite{KMX2} for discussion of the proof of Theorem~\ref{thm_nonrigid_intro}.    

We now turn to Theorem~\ref{th:pull_back_intro}.  

The proof of  (\ref{eqn_sobolev_pullback_theorem_rn}) in the Euclidean case follows from a straightforward mollification argument, so it is natural to try to use mollification to prove analogous results in the Carnot group setting.   The standard approach to mollifying  a map $f:G\ra G'$ would be as follows.  Let $\si_1$ be a symmetric compactly supported smooth probability measure on $G$.  For $\rho\in(0,\infty)$, $x\in G$, let  $\si_\rho:=(\de_\rho)_*\si_1$ be the pushforward of $\si_1$ under the Carnot dilation $\de_\rho$, and $\si_{x,\rho}:=(\ell_x)_*\si_\rho$ be the pushforward of $\si_\rho$ under left translation by $x$.  One  defines the mollified mapping $f_\rho:G\ra G'$ by letting $f_\rho(x)$ to be the average of the pushforward measure $f_*(\si_{x,\rho})$, where the average is defined by treating $G'$ as a linear space, for instance by identifying $G'$ with its Lie algebra via the exponential map.  Unfortunately, such a mollification introduces uncontrollable error terms, a fact which  was pointed out in \cite{vodopyanov_foundations} (however, see \cite{dairbekov_morphism_property_bounded_distortion} for a proof of some analytical results using standard mollification in the special case of the Heisenberg group).   Instead of using standard averaging  to define $f_\rho(x)$ we use a center of mass construction due to Buser-Karcher which is compatible with left translation and Carnot dilation \cite{karcher_buser_almost_flat_manifolds}.  This avoids both the uncontrollable error terms and the limitation to forms of codegree $1$ faced by the method of \cite{vodopyanov_foundations}. 
The main consequence is the following approximation theorem, which is the central technical result in this paper:
\begin{theorem}
\label{thm_approximation_intro}
Let $G$, $G'$ be Carnot groups where $G$ has dimension $N$ and homogeneous dimension $\nu$.
Let $f:G\supset U\ra U'\subset G'$ be a $W^{1,p}_{\loc}$-mapping between open subsets, for some $p>\nu$.  Then there is a family of smooth mappings $\{f_\rho:U_\rho\ra G'\}_{\rho>0}$ with the following properties:
\bit
\item $U_\rho\subset U$ is an open subset such that every compact subset $K\subset U$ is contained in $U_\rho$ when $\rho$ is sufficiently small;
\item $f_\rho\ra f$ locally uniformly;
\item If $\om\in \Om^k(U')$, $\eta\in\Om^{N-k}_c(U) $ are smooth differential forms whose weights satisfy $\wt(\om)+\wt(\eta)\leq -\nu$, then 
\begin{equation*}  
  f_\rho^*\om\wedge \eta \stackrel{L^1_{\loc}}{\lra} f_P^*\om\wedge \eta\,,
\end{equation*}
and, in particular
\begin{equation*}  
\int_Uf_P^*\om\wedge \eta=\lim_{\rho\ra 0}\int_Uf_\rho^*\om\wedge \eta\,.
\end{equation*}
\eit
\end{theorem}

Theorem~\ref{th:pull_back_intro} follows immediately from Theorem~\ref{thm_approximation_intro} by Stokes' theorem.  As discussed in Remark~\ref{rem_stronger_version_exists} there is  
 an extension of this result to $p\leq \nu$, see \cite{kmx_approximation_low_p}.

\subsection*{Organization of the paper}
Section \ref{sec_pullback_theorem} gives the statement of the Pullback Theorem and deduces some corollaries.  Section~\ref{sec_center_of_mass_mollification} discusses center of mass and mollification of maps into Carnot group; Section~\ref{sec_pansu_pullback_mollification} gives the proof of the Pullback Theorem.  Sections~\ref{sec_rigidity_products} and \ref{sec_complex_heisenberg_setup} deal with product structure and complexified groups, respectively.

\bigskip
{\bf Acknowledgements.}   We would like to thank Mario Bonk for suggesting a rigidity argument using differential forms, which simplified an earlier proof based on currents.   We would also like to thank Misha Kapovich for pointing out that the center of mass construction we were using had been discovered long ago by Buser-Karcher.

\section{Preliminaries}
\label{sec_prelim}
 In this section we collect some background material and fix notation and conventions.

\subsection{Carnot groups} 
\label{subsec_carnot_groups} 
We recommend the survey \cite{ledonne_primer_carnot_groups} for a discussion of the material here.

Without explicit mention  we 
will equip all Lie groups with Haar measure.  As is customary, we will conflate two standard definitions of the Lie algebra of a Lie group $G$ -- the tangent space at the identity $T_eG$ and the space of left invariant vector fields; likewise we will conflate a linear subspace of $\fg$ with the corresponding left invariant subbundle of the tangent bundle $TG$.  We will typically use the same notation for a Lie group homomorphism $G\ra G'$ and the induced homomorphism of Lie algebras $\fg\ra \fg'$.

\begin{definition}
A {\bf Carnot group} is a simply connected Lie group $G$ together with a decomposition of its Lie algebra $\fg$ into linear subspaces 
$$
\fg=V_1\oplus\ldots\oplus V_s
$$
such that:
\bit
\item $[V_i,V_j]\subset V_{i+j}$ for $1\leq i,j\leq s$.
\item The {\bf horizontal space (or first layer)} $V_1$ generates $\fg$.
\eit
A graded homomorphism between  Carnot groups $(G,\fg=\oplus_{j=1}^sV_j)$, $(G',\fg'=\oplus_{k=1}^{s'}V_k')$ is a Lie group homomorphism $G\ra G'$ such that the induced homomorphism of Lie algebras $\fg\ra\fg'$ carries $V_j$ into $V'_j$ for all $1\leq j\leq s$.   We will use $\aut(G)$ to denote the group of graded automorphisms of a Carnot group.
\end{definition}

Properties of a Carnot group $(G,\fg=\oplus_jV_j)$:
\bit
\item $G$ is nilpotent, and $V_{j+1}=[V_1,V_j]$ for all $1\leq j\leq s$.
\item  There is a $1$-parameter group of Carnot group automorphisms $\{\de_r:G\ra G\}$ such that $\de_r$ scales $V_j$ by the factor $r^j$.  We refer to these automorphisms as (Carnot) dilations.
\item Without explicit mention, we will assume that the first layer $V_1\subset \fg$ is equipped with an inner product, and extend this to a left invariant Riemannian metric on the left invariant subbundle of $TG$ determined by $V_1$.
\item The  Carnot-Caratheodory distance $d_{CC}$ between  $x_0,x_1\in G$ is the infimal length of a $C^1$-path $\ga:[0,1]\ra G$ with $\ga(i)=x_i$, $\ga'(t)\in V_1$ for all $t\in [0,1]$, and the length is taken with respect to the given Riemannian metric on $V_1$.
\item The topology induced by the metric $d_{CC}$ is the same as the topology coming from $G$ as a manifold.
\item $d_{CC}$ is invariant under left translation, and scaled by dilations: 
$$
d_{CC}(\de_r(x),\de_r(x'))=rd_{CC}(x,x')\,.
$$
\item The {\bf homogeneous dimension} of $G$ is $\nu:=\sum_jj\dim V_j$; this is the Hausdorff dimension of $(G,d_{CC})$, and there is a constant $C\neq 0$ such that for every $x\in G$, $r\in [0,\infty)$, the $\nu$-dimensional Hausdorff measure of the $r$-ball satisfies: 
$$
\h^\nu(B(x,r))=Cr^\nu\,.
$$
\eit

\bigskip
\begin{definition}
The {\bf $n^{th}$ real Heisenberg group $\H_n$} is the Carnot group with graded Lie algebra  $\fh_n=V_1\oplus V_2$, where $V_1=\R^{2n}$, $V_2=\R$, and the bracket $X,Y\in V_1$ is given by  $[X,Y]=-\om(X,Y)\in V_2$; here $\om=\sum_idx_{2i-1}\wedge dx_{2i}$ is the standard symplectic form on $\R^{2n}$.

The {\bf $n^{th}$ complex Heisenberg group $\H_n^\C$} is the complexification of $\H_n$; its graded Lie algebra is $\fh_n^\C:=\fh_n\otimes\C$; more concretely, it has the grading $\fh_n^\C=V_1^\C\oplus V_2^\C$, where $V_1^\C=V_1\otimes\C=\C^{2n}$, $V_2^\C=V_2\otimes\R=\C$, and for $X,Y\in V_1^\C$ the bracket is given by $[X,Y]=-\om^\C(X,Y)\in \C=V_2^\C$, and $\om^\C$ is the extension of $\om$ to a $\C$-bilinear form $V_1^\C\times V_1^\C\ra \C$.  
\end{definition}

\bigskip

\subsection{Sobolev mappings}
Starting from the mid-90s the theory of Sobolev spaces in metric measure spaces advanced rapidly, so that basic results were established in great generality.  However in this paper we only need results in the Carnot group case, where they were (almost all) already known, and in most cases technically simpler to establish.  We refer the reader to \cite{pansu,koranyi_reimann_foundations_quasiconformal_mappings_heisenberg_group,margulis_mostow_differential_quasiconformal_mapping,heinonen_lectures_analysis_metric_spaces,hajlasz_koskela_sobolev_met_poincare,vodopyanov_foundations}.

Let $G$, $G'$ be Carnot groups, where $G$ has homogeneous dimension $\nu$ and grading $\fg=\oplus_jV_j$.

\begin{definition}
\label{def_sobolev_space}
If $1\leq q\leq \infty$, $U\subset G$ is open, then $W^{1,q}_{\loc}(U)$ is the set of elements $u\in L^q_{\loc}(U)$ such that for every horizontal left invariant vector field $X\in V_1$,  distributional directional derivative $Xu$ belongs to $L^q_{\loc}(U)$. We let $|D_hu|\in L^q_{\loc}(U)$ be the pointwise supremum of the directional derivatives $Xu$, where $X\in V_1$ and $|X|=1$. 
\end{definition}
We remark that a variety of other definitions of Sobolev spaces have been developed (see for instance  \cite{hajlasz_sobolev_spaces,cheeger_differentiability,shanmugalingam_newtonian}) and they are all equivalent to Definition~\ref{def_sobolev_space} in the case of Carnot groups equipped with Haar measure.

\begin{lemma}
\label{lem_pi_sobolev} 
\cite[Theorem 1.11]{garofalo_nhieu_1996}
There is a  constant  $C_q<\infty$  with the following property. If $x\in G$, $r>0$, $u\in W^{1,q}(B(x,  r))$ with $q> \nu$
then  $u$ has a continuous representative, and it satisfies
\begin{equation}
\label{eqn_sobolev_inequality}
|u(y_1)-u(y_2)|  \leq   C_q r^{\frac{\nu}{s}}d(y_1,y_2)^{1-\frac{\nu}{q}}\left(\av_{B(x, r)}|D_hu|^q\,d\mu\right)^\frac{1}{q}
\end{equation}
for all $y_1,y_2\in B(x,r)$, where $\av$ denotes the average.
\end{lemma}

\begin{proof}
Indeed it follows  from  \cite[Theorem 1.11 and (3.19)]{garofalo_nhieu_1996}  that the assertion holds for $x=e$ and $r=1$. The general case
follows by left-translation and scaling. 
\end{proof}

\bigskip
\begin{definition}
\label{def_sobolev_mapping}
Suppose $1\leq q\leq \infty$ and $U\subset G$ is open.  
Then a measurable mapping $f:U\ra G'$ is in $W^{1,q}_{\loc}$ if there exists $g\in L^q_{\loc}(U)$ such that
 for every $z\in G'$, the composition $d_z\circ f:U\ra \R$ belongs to $W^{1,q}_{\loc}(U)$
  and its horizontal differential satisfies $|D_h(d_z\circ f)|\leq g$ almost everywhere.  
  Here $d_z:=d_{CC}(z,\cdot)$ denotes the distance function from $z$.
\end{definition}

\bigskip
\begin{lemma}
\label{lem_sobolev_embedding} Let $q > \nu$. 
There is a constant $C_q$ with the following property. 
If $f \in W^{1,q}(B(x,r); G')$ and if $g$ is as in Definition~\ref{def_sobolev_mapping}
then $f$ has a continuous representative and it satisfies
\begin{equation}  \label{eq:sobolev_estimate_Gprime}
d_{CC}(f(y),f(x))\le C_q \, 
 r\left(  \av_{B(x, r)}    |g|^q\,d\mu\right)^\frac{1}{q}\,.
\end{equation}
\end{lemma}

\begin{proof}
The main point is to show that $f$ has a continuous representative. 
The estimate \eqref{eq:sobolev_estimate_Gprime} then 
 follows from Lemma~\ref{lem_pi_sobolev} and 
Definition~\ref{def_sobolev_mapping} since $d_{CC}(f(y),f(x))=(d_{f(x)}\circ f)(y)$.

To see that $f$ has a continuous representative, let  $D$ be  a countable dense subset $D\subset G'$.
Then there exists   a null set $N\subset U$ such that \eqref{eqn_sobolev_inequality} holds for the composition $d_z\circ f$ for every $z \in D$
 and every $y_1,y_2\in B(x,r) \setminus N$.  In particular, the collection $\{d_z\circ f\}_{z\in D}$ is uniformly H\"older continuous on $B(x,r) \setminus N$. 
  Since $B(x,r) \setminus N$ is dense in $B(x,r)$ and $G'$ is complete, there
  exists a unique continuous extension $\bar f : B(x,r) \to G'$  of 
   $f\restr_{B\setminus N}$ which agrees with $f$ almost everywhere in $B$. 
\end{proof}

\bigskip
Let $U \subset G$ be open, let $q > \nu$ and let $f \in W^{1,q}_{\rm loc}(U;G')$. Then 
it follows from  Lemma~\ref{lem_sobolev_embedding} that $f$ has a continuous representative. To see this, write  $U$ as a countable union of open balls
$B_i$ such that their closure $\overline B_i$ are contained in $U$. By Lemma~\ref{lem_sobolev_embedding} there exist continuous functions $\overline{f}_i : B_i \to G'$
such that $\overline{f}_i = f$ a.e.\ in $B_i$. Since the functions  $\overline{f}_i$ are continuous it follows that $\overline f_i =\overline  f_j$ on $B_i \cap B_j$ if $B_i \cap B_j \neq \emptyset$. 
Thus there exists a continuous function $\overline  f: U \to G'$ with  $f\restr_{B_i} = \overline{f}_i$. In particular $\overline f = f$ almost everywhere.

\bigskip

Fix a mapping $f:G\supset U\ra G'$, where $U\subset G$ is open.   For all $x\in U$, $0<r<\infty$, define
$$
f_x=\ell_{f(x)^{-1}}\circ f\circ\ell_x\,,\qquad f_{x,r}=\de_{r^{-1}}\circ f_x\circ \de_{r}\,,
$$
where $\ell_h$, $\de_r$ denote left translation by $h$ and Carnot dilation, respectively. 

\begin{definition}
$f$ is {\bf Pansu differentiable} at $x\in U$ if there is a graded group homomorphism $\Phi:G\ra G'$ such that 
$$
f_{x,r}\lra \Phi
$$
uniformly on compact sets as $r\ra 0$.  We refer to $\Phi$, as well as the associated homomorphism of graded Lie algebras $\fg\ra\fg'$ as the {\bf Pansu differential} of $f$ at $x$.  We use $D_Pf(x)$ to denote both objects $D_Pf(x):G\ra G'$, $D_Pf(x):\fg\ra \fg'$.
\end{definition}

\begin{theorem}
\label{thm_pansu_differentiability} (\cite[Corollary 1]{vodopyanov_pansu_differentiability}; see also \cite{pansu,margulis_mostow_differential_quasiconformal_mapping}) 
If $f:U\ra G'$ is a $W^{1,q}_{\loc}$ mapping for some $q>\nu$, then $f$ is Pansu differentiable almost everywhere.
\end{theorem}

\noindent
Another proof of Pansu differentiability may be found in  \cite{kmx_approximation_low_p}.

\bigskip
\subsection{Quasiconformal and quasisymmetric mappings}
We let $G$, $G'$ denote Carnot groups as before, where $G$ has homogeneous dimension $\nu>1$.
 
\begin{definition}
Let $f:X\ra Y$  be a homeomorphism between metric spaces, and $\eta:[0,\infty)\ra [0,\infty)$ be a homeomorphism.     For $x\in X$, let
$$
H_f(x):=\limsup_{r\ra 0}\frac{\sup\{d(f(x),f(y))\;|\;d(x,y)\leq r\}}{\inf\{d(f(x),f(y))\;|\; d(x,y)\geq r\}}\,,
$$
and $H_f:=\sup_xH_f(x)$. 
Then $f$  is   {\bf quasiconformal} if $H_f<\infty$, and $f$ is  {\bf $\eta$-quasisymmetric} if for any distinct points $p,x,y\in X$, 
$$
\frac{d(f(p),f(x))}{d(f(p),f(y))}\leq \eta\left(\frac{d(p,x)}{d(p,y)}  \right)\,.
$$
\end{definition}

\bigskip
If $f$ is an $\eta$-quasisymmetric homeomorphism then its inverse is an $\tilde \eta$-quasisymmetric homeomorphism 
with $\tilde \eta(t) = \frac{1}{ \eta^{-1}(t^{-1})}$, and   if  $f$ is an $\eta_f$ quasisymmetric homeomorphism and $g$ is an $\eta_g$ quasisymmetric homeomorphism
then $f \circ g$ is an $\eta_f \circ \eta_g$ quasisymmetric homeomorphism, see \cite[Proposition 10.6]{heinonen_lectures_analysis_metric_spaces}.

We will need the following properties of quasiconformal and quasisymmetric homeomorphisms:
\begin{theorem}[See \cite{heinonen_koskela_definitions_quasiconformality,heinonen_koskela,heinonen_lectures_analysis_metric_spaces}]
\label{thm_properties_qc_homeomorphisms}
Let $f:G\supset U\ra U'\subset G'$ be a homeomorphism.  Then
\ben
\item If $f$ is quasiconformal, then it is locally $\eta$-quasisymmetric for some $\eta=\eta(H_f)$, i.e. for every $x\in U$, there is an $r=r(x)>0$ such that $f\restr_{B(x,r)}:B(x,r)\ra f(B(x,r))$ is an $\eta$-quasisymmetric homeomorphism.
\item Every global quasiconformal homeomorphism $f: G\ra G'$ is $\eta=\eta(H_f)$-quasisymmetric.
\item If $f$ is quasiconformal, then $f\in W^{1,p}_{\loc}$ for some $p=p(H_f)>\nu$.
\item If $f$ is quasiconformal, then  $D_Pf(x)$ is an isomorphism for a.e. $x$, and the sign of the determinant of $D_Pf(x)$ is locally constant, and agrees with the local degree of $f$.
\een
\end{theorem}
\begin{proof}
For (1) and (2) see Remark 1.8(b) and Theorem 1.7 in \cite{heinonen_koskela_definitions_quasiconformality}.

(3).  This is based on Gehring's lemma, and is treated in Section 7 of \cite{heinonen_koskela}.  However, in the Carnot group case, one may     use
   the differentiability and absolute continuity results from \cite{pansu} or  \cite{margulis_mostow_differential_quasiconformal_mapping} and Gehring's lemma instead.

(4).  Since $f$ is a homeomorphism, so is $f_{x,r}$ for every $r\in(0,\infty)$.  If $x$ is a point of Pansu differentiability, then $f_{x,r}\ra D_Pf(x)$ uniformly on compact sets, and therefore its local degree near $e\in G$ is constant for small $r$, and coincide with the local degree of $D_Pf(x)$, which is the sign of the determinant of $D_Pf(x):\fg\ra\fg'$.  This implies the local constancy of the sign of $\det(D_Pf(x))$.

\end{proof}

\bigskip

\bigskip
\section{Differential forms on Carnot groups}
\label{se:weights}
The purpose of this section is to discuss two key notions for the de Rham complex on Carnot groups: the weight of  a differential form (which reflects the action of the dilation map $\delta_r$) and the Pansu pullback; this sets the stage for the Pullback Theorem in the next section, which  asserts that the distributional exterior derivative commutes with Pansu pullback, but only if the form which is pulled back and the test form satisfy certain restriction on their weights.

Let $G$ be a Carnot group of step $s$ with Lie algebra $\fg$.  
  We retain the notation from before: 
\begin{itemize}
\item  
$N$ and $\nu$ are the dimension and homogeneous dimension  of $G$ respectively. 
\item $\de_r:G\ra G$ is dilation by $r$; we use the same notation for the dilation on $\fg$.
\item Let $\{X_i\}_{1\leq i \leq N}$ be a graded basis 
of $\fg$, i.e., each $X_i$ lies in one of the layers $V_1, \ldots, V_s$, and let $\{\th_i\}_{1\leq i\leq N}$ be the dual graded basis of $\fg^*=\Lambda^1 \fg$.
\item If $I=\{i_1,\ldots,i_k\}\subset \{1,\ldots,N\}$, then $\th_I:=\th_{i_1}\wedge\ldots\wedge \th_{i_k}$.
\end{itemize}
As usual, we will identify $X_i$ and $\theta_I$ with the corresponding left invariant vector fields and forms on $G$ when convenient.

We first define weights on  the graded algebras of  multi-vectors $\Lambda_* \fg = \oplus_{k=1}^N \Lambda_k \fg$ and alternating forms
$\Lambda^* \fg = \oplus_{k=1}^N \Lambda^k \fg$. 
\begin{definition}
A  $k$-form $\al  \in \Lambda^k \fg$ or  
$k$-vector  $\xi\in \Lambda_k(\fg)$ is {\bf homogeneous with weight $w$} if 
$(\de_r)_*\al=r^w\,\al$ or $(\de_r)_*\xi=r^w\,\xi$, respectively, for all
$r\in(0,\infty)$. 
\end{definition} 
Note that the weight of a non-vanishing homogeneous $k$-vector or $k$-form is uniquely defined. We denote it by  $\wt(\xi)$ and  $\wt(\alpha)$.
Note also  that the zero form and the zero vector are homogeneous with any weight. Hence they do not have a well-defined weight. 

\begin{remark}
Most of the facts about weights and homogeneity discussed below follow the behavior of diagonalizable actions under dualizing, tensor products, and contraction.  However, we give a calculation based treatment.
\end{remark}
  If we denote the elements of a homogeneous basis of $\fg = \oplus_{i=1}^s V_i$  by $X_{i,j}$ with $X_{i,j} \in V_i$, $1 \le i \le s$ and $ 1 \le j \le \dim V_i$,
then $\wt(X_{i,j}) = i$. If $\theta_{i,j}$ denotes the dual basis of one forms then $\wt(\theta_{i,j}) = -i$. 
By direct inspection we see that  for a multiindex $I = ((i_1, j_1), \ldots, (i_k, j_k))$ with $(i_\ell, j_\ell) \ne (i_m, j_m)$ for $\ell \ne m$ we have
\begin{equation}  \label{eq:weights_basis_fg}
\wt(\theta_I) = - \sum_{\ell =1}^k i_\ell, \quad  \wt(X_I) =  \sum_{\ell =1}^k i_\ell.
\end{equation}
It follows from  \eqref{eq:weights_basis_fg} that
every $k$-form $\alpha \in \Lambda^k(\fg)$ has a canonical decomposition
$$
\alpha=\sum_{-\nu\leq j\leq -k} \alpha_j
$$
where $\alpha_j$  is  zero or  a homogeneous $k$-form of weight $j$.
 Likewise every $k$-vector $X$ has a decomposition
$$
X = \sum_{k\leq j\leq \nu}X_j
$$
where $X_j$ is homogeneous of weight $j$ or zero.   By definition, the weight of a $k$-form or $k$-vector 
is the maximal weight of its homogeneous nonvanishing summands.

\begin{lemma}
\label{lem_weight_facts}

\mbox{}
\ben
\item   \label{it:weight_facts_vectorspace}    The sets
$$\{ \alpha \in \Lambda^k\fg : \alpha \ne 0, \wt(\alpha) \le w  \} \cup \{0\}$$
and 
 $$\{X \in \Lambda_k\fg : \xi \ne 0, \wt(X) \le w\} \cup \{0\}$$
are vector spaces.

\item  \label{it:weight_facts_products}    If $\al, \beta \in \Lambda^*\fg$ and $X, Y \in\Lambda_*\fg$ are
homogeneous, then so are
$\al  \wedge\beta$, $X\wedge Y$, and the interior product $i_{X} \alpha$.
Moreover 
\begin{align*}
\wt(\al\wedge\beta)=&\wt(\al)+\wt(\beta)\\
\wt(X \wedge Y)=&\wt(X )+\wt(Y)\\
 \wt(i_{X} \al) =& \wt(\al) +  \wt(X) 
\end{align*}
as long as the form or vector on the left hand side does not vanish. 
\item  \label{it:weight_facts_pullback}  If $\gamma \in \Lambda^k\fg'$ is homogeneous and
$\Phi:\fg \ra \fg'$ is a graded algebra homomorphism, then
$\Phi^*\gamma  $ is homogeneous with
$\wt(\Phi^*\ga)=\wt(\ga)$ as long as $\Phi^*\ga \ne 0$. 
\item If $\Phi:\fg\ra \fg'$ is a graded homomorphism, $\al,\be\in \La^*\fg$, $\ga\in \La^*\fg'$,
\begin{align}
\wt(\al\wedge\be)\le &\wt(\al)+\wt(\beta)  
      \label{eq:subadditivity_wedge_weight} \\
      \wt(d\alpha) \le &\wt(\alpha)  \quad \forall \al \in \Lambda^*\fg,\\
 \wt(\Phi^* \gamma) \le &   \wt(\gamma)  \quad \hbox{$\forall \gamma \in \Lambda^* \fg'$} 
       \label{eq:pullback_nonhomogeneous}
\end{align}
whenever the forms on the left hand side do not vanish. 
Moreover if $\Phi$ is an isomorphism and $\ga\neq 0$ then
\begin{equation}
 \label{eq:pullback_nonhomogeneous_iso}
\begin{aligned}
\quad  \wt(\Phi^* \gamma) =  \wt(\gamma). \\ 
\end{aligned}
\end{equation}
  \item    \label{it:weight_facts_vanishing}   If $\al \in \Lambda^k \fg$ then
\begin{align*}
\alpha \ne 0 \quad \Longrightarrow  \quad  \wt(\alpha) \ge -k + N - \nu.
\end{align*}

\item \label{item_wrong_weights_pair_zero}
If $\al\in\Om^k(G)$, $\xi\in \Om_k(G)$ are homogeneous and
$\wt(\al)+\wt(\xi)\neq 0$ then $\al(\xi)\equiv 0$.
\item  \label{it:weight_facts_dalpha} 
 If $\alpha \in \Lambda^*\fg$ is homogeneous  and $d\alpha \ne 0$ then
\begin{align*}
\wt(d\alpha) = \wt(\alpha)
\end{align*}
\een
\end{lemma}

\bigskip
In  assertion \eqref{it:weight_facts_dalpha} of the lemma, the action of the exterior derivative on $\Lambda^*\fg$ is defined by extending the elements of $\Lambda^*\fg$ to
left invariant forms, 
\begin{equation}  \label{eq:exterior_derivative_on_algebra}
\begin{aligned}
d\alpha&(X_0, \ldots, X_k) \\
&= \sum_{0 \le i < j \le k }  (-1)^{i+j} \alpha([X_i, X_j], X_0, \ldots, \widehat X_i, \ldots, \widehat X_j, \ldots, X_k).
\end{aligned}
\end{equation}
 See \cite[Lemma 14.14]{Michor}.

 \bigskip

\begin{proof} 
Assertion  \eqref{it:weight_facts_vectorspace} follows directly from the definition of the weight. 
The first two identities in   \eqref{it:weight_facts_products}   
 also follow immediately from the definition and the fact that pullback und pushforward
under $\delta_r$ commute with the wedge product. 
The third identity follows from a short calculation using $(\delta_r)_* i_X \alpha(Z) = \delta_{r^{-1}}^* i_X \alpha(Z) = (i_X \alpha) ( (\delta_{r^{-1}})_* Z)$. 

 To show  \eqref{it:weight_facts_pullback} we use that  $(\delta_r)_* = \delta_{r^{-1}}^*$ and  that graded algebra homomorphisms commute with the dilation map.

Assertion (4) follows from assertions (2) and (3) and the definition of the weight.

Assertion \eqref{it:weight_facts_vanishing} is clear for $k=N$ since the volume form has weight $-\nu$. 
Let $k < N$ and assume that $\alpha \in \Lambda^k \fg$ with $\alpha \ne 0$ and $\wt(\alpha) < -\nu + N -k$. 
 Every one-form has weight $\le -1$. Thus by  \eqref{eq:subadditivity_wedge_weight} 
 we get $\wt(\alpha \wedge \beta) < -\nu$ for every $k$-form $\beta$. Hence $\alpha \wedge \beta = 0$ for all $k$-forms $\beta$. 
 This implies $\alpha = 0$ and yields a contradiction.

Assertion  (\ref{item_wrong_weights_pair_zero}) follows from the fact that pullback commutes with interior product.
 
 Finally \eqref{it:weight_facts_dalpha} follows from  \eqref{eq:exterior_derivative_on_algebra} 
 as well as the identities $(\delta_r)_* = \delta_{r^{-1}}^*$ and $(\delta_{r^{-1}})_*[X_i, X_j] = [(\delta_{r^{-1}})_*X_i,
 (\delta_{r^{-1}})_*X_j]$.

\end{proof}

\medskip

We now consider differential forms $\alpha$ defined on an open subset $U$ of a Carnot group $G$. 
We denote the differential forms on $U$ by $\Omega^*(U)$ and the
   compactly supported forms by $\Omega^*_c(U)$.

\bigskip

Identifying the   left invariant differential forms with $\Lambda^*\fg$ we can write a $k$-form
$\alpha \in \Omega^k(U)$  as
\begin{align*}
\alpha = \sum_{I}  a_I \theta_I  \quad \hbox{where} \quad a_I: U \to \R
\end{align*}
 where the sum runs over ordered multiindices of length $k$. The functions $a_I$ are determined uniquely. 
We say that $\alpha$ is a  measurable (or distributional,  continuous, smooth,  \ldots) form if the coefficient functions  $a_I$ 
are measurable (or distributions,  continuous, smooth, \ldots). We say that $\alpha$ vanishes a.e.\ if all functions $a_I$ vanish a.e.
For fixed $x \in U$ we can view $\alpha(x) = \sum_I a_I(x) \theta_I$ as an element of $\Lambda^k \fg$ and we define its weight as above.

\begin{definition} If $\alpha$ is measurable and does not vanish a.e.\ in $U$ we define
\begin{equation}
\wt(\alpha) := \esssup_{x \in U: \alpha(x) \ne 0} \wt(\alpha(x)).
\end{equation}
If $E \subset U$ is measurable and the restriction of $\alpha$ to $E$ does not vanish a.e.\
 then we define
\begin{equation}
\wt(\alpha_{|E} ) := \esssup_{x \in E: \alpha(x) \ne 0} \wt(\alpha(x)).
\end{equation}
\end{definition}

It follows from    \eqref{eq:subadditivity_wedge_weight} that
\begin{equation}
\wt(\alpha \wedge \beta) \le \wt(\alpha) + \wt(\beta)  \quad \forall \alpha, \beta \in \Omega^*(U)
\end{equation}
whenever $\alpha \wedge \beta$ does not vanish a.e.

\medskip

We now come to a key object in our analysis, the Pansu pullback of a differential form under map which is a.e.\  Pansu differentiable.
Let  $G$ and $G'$ be Carnot groups, where $G$ has dimension $N$ and homogeneous dimension $\nu$.  
  Suppose $f:G\supset U\ra U'\subset G'$ is   a $W^{1,p}_{\loc}$  mapping between open sets for some $p>\nu$.
 Then $f$ is Pansu differentiable at almost every $x \in U$.  The Pansu differential  
 $D_Pf(x)$ can be viewed as  a graded homomorphism from $G$ to $G'$  or from $\fg$ to $\fg'$ (here we identify a graded group homomorphism with its
 tangent map). 
 \begin{definition}   Let $\alpha$ be a continuous   differential form on $U'$. The Pansu pullback $f_P^*\alpha$ is defined by
 \begin{equation}
 (f_P^*\alpha)(x) = \sum_{I} (a_I \circ f)(x)   \, \,  (D_Pf(x))^* \theta_I.
 \end{equation}
 \end{definition}

 \bigskip
 
 \bigskip

\bigskip

\section{The Pullback Theorem}
\label{sec_pullback_theorem}
In this section we present the key technical result in this paper --  the Pullback Theorem -- and deduce several corollaries.   

We begin by pointing out that in general, Pansu pullback does not commute with the exterior derivative, even for smooth contact diffeomorphisms.

\begin{lemma} 
\label{lem_pullback_exterior_derivative} 
Let $f:\H\supset U\ra U'\subset \H$ be a smooth contact diffeomorphism between open subsets of the Heisenberg group. If for some $k\in \{0,1,2\}$ the Pansu pullback $f_P^*:\Om^k(U')\ra \Om^k(U)$ commutes with the exterior derivative, then $f$ preserves the foliation by cosets of the center, i.e. $Df(V_2)=V_2$. 
\end{lemma}
\begin{proof}
We will show that for every $u\in \Om^0(U')$ such that 
$X_3u= 0$ we have $X_3(u\circ f)=0$.  This implies $Df(V_2)=V_2$.

Take $k=0$.   Then since $f_P^*u:=u\circ f$, by assumption 
$$
X_3(u\circ f)=(d(u\circ f))(X_3)=(f_P^*du)(X_3)=0
$$
because $D_Pf(x)(X_3)\in V_2$ since $D_Pf(x)$ is a graded homomorphism.  

Now suppose $k=1$.   Let $\om:=du$.  Then $d\om=0$, so by assumption we have $df_P^*\om=0$.  Hence $f_P^*\om=d\hat u$ for some $\hat u\in\Om^0(U)$ with $X_3\hat u=0$.  We have $d\hat u\restr_{V_1}=d(u\circ f)\restr_{V_1}$ because Pansu pullback agrees with ordinary pullback on $V_1$.  It follows that $d(\hat u-u\circ f)\restr_{V_1}=0$, which implies that $\hat u-u\circ f$ is constant.  Therefore $X_3(u\circ f)=X_3\hat u=0$. 

Take $k=2$.  Let $\om:=u\,\th_1\we\th_2$, so $d\om=(X_3u)\th_1\we\th_2\we\th_3=0$.  Then  $f_P^*\om=(u\circ f)\,\al \,\th_1\we\th_2$, where $\al:=\det[D_Pf\restr_{V_1}]$ (which is nowhere zero since $f$ is a diffeomorphism), so by assumption we have
$$
0=df_P^*\om=[X_3((u\circ f)\al)]\,\th_1\we\th_2\we\th_3\,
$$
and therefore $(X_3\al)\cdot (u\circ f)+\al\,(X_3(u\circ f))=0$.  This also holds if we replace $u$ by $u+c$ for any $c\in \R$, forcing $X_3(u\circ f)=0$. 
\end{proof}

\bigskip

\begin{theorem}[Pullback theorem]  \label{th:pull_back} 
Let $G$ and $G'$ be Carnot groups, where $G$ has dimension $N$ and homogeneous dimension $\nu$, and suppose $f:G\supset U\ra U'\subset G'$ is a $W^{1,p}_{\loc}$  mapping between open sets for some $p>\nu$.
  Suppose  $\om\in \Om^k(U')$ is a continuous form with continuous distributional exterior derivative $d\om$, suppose $\eta\in \Om^{N-k-1}_c(U)$ is smooth, and 
\begin{equation}
\label{eqn_om_eta_weight_condition}
\wt(\om)+\wt(d\eta)\leq -\nu\,,\quad \wt(d\om)+\wt(\eta)\leq -\nu\,.
\end{equation}
Then
\begin{equation}  \label{eq:pullback_theorem}
\int_U(f_P^*d\om)\we\eta+(-1)^k\int_Uf_P^*\om\we d\eta=0\,.
\end{equation}
\end{theorem}

If one of the forms in  \eqref{eqn_om_eta_weight_condition} vanishes, then we interpret its weight as $-\infty$, so that the corresponding inequality holds.

Note that by the Sobolev embedding theorem for Carnot groups, a $W^{1,p}_{\loc}$ mapping is continuous when $p>\nu$, see Lemma~\ref{lem_sobolev_embedding}.

We will prove Theorem~\ref{th:pull_back} by a new approximation of $f$ by smooth maps $f_\rho$ which respects the structure of the Carnot group. 
The approximation is introduced in Section~\ref{sec_center_of_mass_mollification} and the  proof of Theorem~\ref{th:pull_back}
is carried out in Section~\ref{sec_pansu_pullback_mollification}.

 We will usually use the following special case of  Theorem~\ref{th:pull_back}:
\begin{theorem}[Pullback Theorem (special case)]  \label{co:pull_back2}
Let $G,G',U,U'$ and $f$ be as in Theorem~\ref{th:pull_back}.
Suppose $\varphi \in C_c^\infty(U)$ and  that $\alpha$ and $\beta$ are closed left invariant forms
which satisfy 
\begin{equation} \label{eq:weight_special}
\deg \alpha + \deg \beta = N -1 \quad \hbox{and} \quad  \wt(\alpha) + \wt(\beta) \le -\nu + 1.
\end{equation}
Then 
\begin{equation} \label{eq:pull_back_identity_special}
 \int_U f_P^*(\alpha) \wedge d(\varphi \beta) = 0.
\end{equation}
\end{theorem}
\begin{proof}
Letting $\eta = \varphi \beta$, we have $d\eta = d \varphi \wedge \beta$.  Thus
 $\wt(d\eta) \le \wt(\beta) - 1$. Hence the assertion follows from 
Theorem~\ref{th:pull_back}.
\end{proof}

\bigskip\bigskip
For a closed form  $\beta$ we have $d(\varphi \beta) = d\varphi \wedge \beta$. 
Thus the identity \eqref{eq:pull_back_identity_special} implies that the codegree $1$ form  $f_P^*(\alpha) \wedge \beta$ is distributionally closed, i.e. its  distributional exterior derivative
vanishes.   Letting $\al$ and $\be$ vary, we obtain conditions on the distributional   derivatives of certain wedge powers of the Pansu differential $D_Pf(x)$.  We will study these conditions more systematically in a subsequent paper;    our focus here will be on making judicious choices of forms to deduce rigidity for mappings.

Interestingly, Theorem~\ref{co:pull_back2} yields no information on derivatives of the Pansu differential if the form $\alpha$ or $\beta$ is exact.
More precisely we have the following result:
\begin{lemma}  \label{le:trivial_pull_back}   Assume that $f: U \subset G \to G'$ is Pansu differentiable a.e. in $U$.
Let $\alpha$ and $\beta$ be closed left invariant forms
which satisfy 
\begin{equation} \label{eq:weight_special2}
\deg \alpha + \deg \beta = N -1 \quad \hbox{and} \quad  \wt(\alpha) + \wt(\beta) \le -\nu + 1.
\end{equation}
  and assume
 that there exists a left invariant  form $\gamma$ such that 
$\alpha = d \gamma$ (respectively $\beta = d \gamma$). Then
\begin{equation} \label{eq:trivial_pull_back}
f_P^*(\alpha) \wedge \beta = 0 \quad \hbox{a.e.}
\end{equation}
\end{lemma} 

\begin{remark}
 Since $d(\varphi\beta)=   d \varphi \wedge \beta$,  
    (\ref{eq:trivial_pull_back})  implies  
     $\int_U f_P^*(\alpha) \wedge d(\varphi \beta) = 0$.
\end{remark}

\begin{proof} 
In view of the definition of the Pansu pullback it suffices to show the following 
pointwise statement for $\alpha \in \Lambda^*\fg'$, $\beta \in \Lambda^*\fg$.
If  \eqref{eq:weight_special2} holds, if $\alpha = d\gamma$ or $\beta = d\gamma$
and if $\Phi: \fg \to \fg'$ is a graded algebra homomorphism
then 
\begin{equation}    \label{eq:trivial_pull_back2}
\Phi^* \alpha \wedge \beta = 0. 
\end{equation}  
Here we use that exterior differentiation of left invariant forms induces a linear operation on $\Lambda^*\fg$ (and on $\Lambda^* \fg'$),
see  \eqref{eq:exterior_derivative_on_algebra}.

Now assume  first that  $\alpha = d\gamma$. Then $\deg \gamma = \deg \alpha -1$. We claim that we may assume that $\wt(\gamma) = \wt(\alpha)$. 
To see this write  $\gamma$ as a sum of homogeneous forms of weight $s$
$$\gamma = \sum_{s= -\nu}^{-\deg\gamma} \gamma_{s}.$$
By Lemma~\ref{lem_weight_facts}~\eqref{it:weight_facts_dalpha} we have  $d \gamma_{s} = 0$ or $\wt(d\gamma_s) = s$. 
By the definition of $\wt(\alpha)$  we thus  must have $d \gamma_s = 0$ if $s > \wt(\alpha)$
and $\gamma_{\wt(\alpha)} \ne 0$. 
Hence $d \gamma = d \tilde \gamma$ where $\tilde \gamma =  \sum_{s= -\nu}^{\wt(\alpha)} \gamma_{s}$
 and $\wt(\tilde \gamma) = \wt(\alpha)$. 

 It follows from  \eqref{eq:exterior_derivative_on_algebra} that the action of $d$ on $\fg'$  and $\fg$ commutes with pullback by $\Phi$. Since
 $d\beta = 0$ we get
$$ \Phi^*(\alpha)  \wedge \beta = d \Phi^*(\gamma)  \wedge \beta = d( \Phi^*(\gamma) \wedge \beta).$$
The form  $\Phi^*(\gamma) \wedge \beta$ has degree $N-2$ and weight $ \le -\nu+1$ and hence vanishes by 
   Lemma~\ref{lem_weight_facts}~\eqref{it:weight_facts_vanishing}.
This proves   \eqref{eq:trivial_pull_back2} if $\alpha = d\gamma$.

If $\beta = d\gamma$ then $\deg \gamma = \deg\beta  -1$ and we may assume that $\wt(\gamma) = \wt(\beta)$. 
Now using again  that pullback by $\Phi$ commutes with $d$
 and the fact that $\alpha$ is closed we get
$ d \Phi^*\alpha = 0$.
Thus
$$ \Phi^* \alpha \wedge \beta = \Phi^* \alpha \wedge d\gamma =
 (-1)^{\deg \alpha}   d(\Phi^* \alpha \wedge \gamma).$$
Now by         \eqref{eq:subadditivity_wedge_weight} and 
     \eqref{eq:pullback_nonhomogeneous} the form     $\Phi^* \alpha \wedge \gamma$ has weight $\le -\nu +1$. Since this form has degree $N-2$
it must vanish. 
\end{proof}

\section{Center of mass and mollification in Carnot groups}
\label{sec_center_of_mass_mollification}

In this section we introduce a mollification procedure for mappings into Carnot groups; this will be used in the proof of the pullback theorem in the next section.   The mollification is based on a notion of center of mass due to Buser-Karcher \cite{karcher_buser_almost_flat_manifolds}.  Since our application requires additional estimates, we give a short self-contained construction of the center of mass; our approach is different from that of Buser-Karcher, which  reduces the nilpotent case to the Riemannian one \cite{karcher77}.  We remark that geometric center-of-mass constructions have played an important role in other developments in geometry, for instance \cite{douady_earle,besson_courtois_gallot}.

Let $G$ and $G'$ be Carnot groups with respective Lie algebras
$\fg$, $\fg'$, topological dimensions
$N$ and $N'$, horizontal dimensions $n$ and $n'$, and 
homogeneous dimensions $\nu$ and $\nu'$. 
     We fix left invariant Riemannian metrics on $G$, $G'$.

\subsection{Center of mass in Carnot groups}
We  define a center-of-mass for a compactly supported probability
measure in a Carnot group; this behaves well with respect to left translation
and rescaling.

Recall that if $G$ is a Carnot group, then the exponential map $\exp:\fg\ra G$ 
is a diffeomorphism; we let $\log:G\ra\fg$ denote the inverse diffeomorphism.
Recall that $\exp$ (and hence also $\log$) are $\aut(G)$-equivariant, where
$\aut(G)$ acts on $G$ in the tautological way, and on $\fg$ by the derivative
action, i.e. for $\phi\in\aut(G)$ we have $\exp(D\phi(v))=\phi(\exp(v))$.
For $x\in G$, let $\log_x=\log\circ \ell_{x^{-1}}$, where $\ell_g$ is the left
translation by $g$.

Suppose $\nu$ is a compactly supported Borel
probability measure on $G$.  We define a function
$C_\nu:G\ra \fg$ by integration:
$$
C_\nu(x)=\int_\fg \id_{\fg}\,(d(\log_x)_*\nu)=\int_G\log_x\,d\nu
=\int_G\log_x(y)\;d\nu(y)\,,
$$
where we are viewing $\log_x$ as a vector-valued function in the latter
two integrals.  Let $\tilde C_\nu:=C_\nu\circ \exp$.

\begin{lemma}
\label{lem_carnot_center_of_mass}
For any compactly supported Borel probability measure $\nu$:
\ben

\item    For any automorphism $\Phi:G\ra G$, $x,y\in G$, we have
\begin{equation}  \label{eq:center_of_mass_push_forward}
\begin{aligned}
C_{\Phi_*\nu}(\Phi(x))&=D\Phi(C_\nu(x)),\\
C_{(\ell_y)_*\nu}(yx)&=C_\nu(x)\,.
\end{aligned}
\end{equation}
\item If $\nu$ is inversion symmetric, $I_*\nu=\nu$, then $C_\nu(e)=0$.   Here $I: G\ra G$ is the map given by  $I(g)=g^{-1}$. 
\item $C_\nu$ is smooth.
\item $C_\nu^{-1}(0)$
contains a unique element, which we denote by $\com(\nu)$.
\item \label{it:lipbound_com} There is a constant  $C_1$ such that if $\spt(\nu)\subset B(x,R)$, then $\com(\nu)\subset B(x,C_1R)$.
\item If $g\in G$, $\Phi\in \aut(G)$, then
$$
\com((\Phi\circ\ell_g)_*\nu)=(\Phi\circ\ell_g)(\com\nu)\,.
$$
\item For every $x_0\in G$, $v\in T_{x_0}G$, we have 
\begin{equation}
\label{eqn_compression_bound}
\|DC_\nu(x_0)(v)\|\geq C^{-1}_2\|v\|
\end{equation}
where $C_2$ depends on $\diam(\{x_0\}\cup \spt(\nu))$.
\een
\end{lemma}

\begin{remark}
Although it is not needed in this paper, it is possible to get a finer description of the function $C_\nu$ and the center of mass.  We show in \cite{kmx_approximation_low_p} that $C_\nu \circ \exp$ is a polynomial of degree $s-1$ if $G$ is a step $s$ Carnot group,
that the equation $C_\nu(\exp X) = Z$ can be solved recursively using the grading of $G$ 
and that $\log \circ \com_\nu$ can be expressed as a polynomial in the multilinear moments $$A_i = \int_\fg L_i(Y, \ldots, Y)  \,  d(\log_* \nu)(Y)\,. $$ 
\end{remark}

\begin{proof}~

(1) and (2) follow from unwinding definitions.

(3).  We have 
\begin{equation}
\label{eqn_tilde_c_nu}
\tilde C_\nu(X)=\int_{\fg}\Psi(X,Y)\,d(\log_*\nu)(Y)
\end{equation}
where $\Psi(X,Y):=\log(\exp(-X)\exp(Y))$.  Since $\Psi$ is smooth and $\nu$ has compact support, it follows that $\tilde C_\nu$ is smooth, and hence $C_\nu=\tilde C_\nu\circ \log$ is smooth as well.  This proves (3).

\bigskip
{\em Claim. There exist $\eps>0$, $r_0>0$ such that if $\spt(\log_*\nu)\subset B(0,\eps)$, then $\tilde C_\nu(X)=0$ has a unique solution in $B(0,r_0)$.}

\begin{proof}[Proof of Claim]
We have $\Psi(0,0)=0$ and $D_X\Psi(0,0)=-\id_{\fg}$.  So for some $r_0>0$ we have $\|D_X\Psi(X,Y)+\id_{\fg}\|<\frac{1}{10}$ when $X  , Y\in B(0,r_0)$, and if $\eps\in (0,r_0)$ is sufficiently small, then  $\|\Psi(0,Y)\|<\frac{r_0}{10}$ for $Y\in B(0,\eps)$.  If $\spt(\log_*\nu)\subset B(0,\eps)$ then \eqref{eqn_tilde_c_nu} gives 
\begin{align}
\label{eqn_tilde_c_nu_estimates}
\|\tilde C_\nu(0)\|<\frac{r_0}{10}\,,\quad \|D\tilde C_\nu(X)+\id_{\fg}\| <\frac{1}{10}
\end{align} 
for all $X\in B(0,r_0)$.  The (proof of the)  Inverse Function Theorem now implies that $\tilde C_\nu(X)=0$ has exactly one solution in $B(0,r_0)$. 
\end{proof} 

\bigskip
(4).  By (1), we may apply a Carnot contracting automorphism to reduce to the case when $\spt(\log_*\nu)\subset B(0,\eps)$.  Hence by the claim there is a solution to the equation $C_\nu(x)=0$ in $\exp(B(0,r_0))$.  Similarly, if $C_\nu(x_1)=C_\nu(x_2)=0\in G$, then we may apply a further contracting automorphism and the claim to conclude that $x_1=x_2$.

(5).    This follows from the claim by rescaling: choose $A\in (1,\infty)$ such that
$$
B_{d_{CC}}(e,A^{-1})\subset\exp(B(0,\eps))\subset\exp(B(0,r_0))\subset B_{d_{CC}}(e,A)\,;
$$
applying (1) and the claim to $(\de_{(AR)^{-1}})_*\nu$ we get $\com\nu\in B_{d_{CC}}(e,C_1R)$
with $C_1:=A^2$.

(6).  This follows from (1).

(7).  By (6), after applying a left translation and a Carnot dilation, we may assume that $x_0=e$ and $\spt(\log_*\nu)\subset B(0,\eps)$, where $\eps$ is as in the claim.  Then by \eqref{eqn_tilde_c_nu_estimates} 
$$
\|DC_\nu(e)(v)\|=\|D(\tilde C_\nu\circ\log)(e)(v)\|\geq \frac{9}{10}\|v\|
$$
which implies \eqref{eqn_compression_bound}.
\end{proof}

\bigskip\bigskip

\subsection{Mollifying maps between Carnot groups}
Let $\si_1$ be a smooth probability measure on $G$ with $\spt(\si_1)
\subset B(e,1)$.  We also assume that $\si_1$ is symmetric under inversion:
$I_*\si_1=\si_1$, where $I(x)=x^{-1}$.
For $x\in G$, $\rho\in (0,\infty)$, let $\si_\rho$, $\si_x$, and $\si_{x,\rho}$
be the pushforwards of $\si_1$ under the the corresponding Carnot scaling
and left translation:
$$
\si_\rho=(\de_\rho)_*\si_1\,,\quad \si_x=(\ell_x)_*\si_1\,,\quad
\si_{x,\rho}=(\ell_x\circ\de_\rho)_*\si_1=(\ell_x)_*(\si_\rho)\,.
$$

Let $U \subset G$ be open. Let $G'$ be another Carnot group and let $f: U \to G'$ be continuous. 
For $\rho > 0$ set
\begin{equation}  \label{eq:U_rho}
U_\rho := \{ x : B(x, \rho) \subset U\} 
\end{equation}
For $x \in U_1$ we have $\spt (\sigma_x) = x + \spt (\sigma) \subset U$.
Thus we may define $f_1:U_1\ra G'$ by
$$
f_1(x)=\com(f_*(\si_{x}))\,,
$$
and $f_\rho:U_\rho \ra G'$ by 
\begin{equation}
\label{eqn_def_f_rho}
f_\rho=\de_\rho\circ(\de_{\rho^{-1}}\circ f\circ \de_\rho)_1\circ \de_{\rho^{-1}}\,.
\end{equation}

\begin{lemma}
\label{lem_moll_prop} 
\mbox{}
\ben
\item For all $\rho\in(0,\infty)$, 
     $x\in U_\rho$ we have
\begin{align*}
\de_{\rho^{-1}}\circ f_\rho\circ\de_\rho
&=(\de_{\rho^{-1}}\circ f\circ \de_\rho)_1\\
f_\rho(x)&=\com f_*(\si_{x,\rho}).
\end{align*}
\item If $x \in U_1$ and $f(B(x,1))\subset B(f(x),R)$, then the derivatives of $\de_{R^{-1}}\circ f_1$ are
controlled near $x$, i.e.  $\|D^i(\de_{R^{-1}}\circ f_1)(y)\|<C=C(i)$ for $y$ close to $x$. In particular, 
$Df_1(x)=D(\de_R\circ (\de_{R^{-1}}\circ f_1)(x)=D\de_R\circ T$ where $\|T\|<C'$ and $C'=C'(G,G')$ is independent of $R$.
\item If $f:G\ra G'$ is a group homomorphism, then $f_1=f$.
\item 
If $\{f_k:U\ra G'\}$ is a sequence of continuous maps, 
and $f_k\ra f_\infty$ in $C^0_{loc}(U)$, then the sequence  of 
mollified maps $\{(f_k)_1\}$ converges in $C^j_{loc}(U_1)$ to $(f_\infty)_1$, for all $j$.
\item   We have $f_\rho \to f$ in $C^0_{loc}(U)$ 
       as $\rho \to 0$. 
\een
\end{lemma}

\begin{proof}
(1).  This is immediate from the definition.

(2).  Note that $\de_{R^{-1}}\circ f_1=(\de_{R^{-1}}\circ f)_1$, so we are reduced to the case
when $R=1$.  

Since the derivative estimates are unchanged if we pre or post-compose $f$ with left
translations, we may assume without loss of generality that $f(e)=e$, and that we want to estimate the derivative of $f_1$ at $e$.  So we are interested in 
$f_1(z)=\com(f_*\si_{z})$ for $z$ close to $e\in G$, which is determined implicitly by 
the equation
$$
C_{f_*\si_z}(x)=0\,.
$$             
Using that $\spt(\si_z) \subset B(e,1)$ for $z$ close to $e$  and that $f$ is defined on $B(e,1)$ we get  
\begin{align*}
C_{f_*\si_z}(x)=&\int_{G'}\log_x\,d(f_*\si_z)\\
=&  \int_{B(e,1)}      (\log_x\circ f)\,d\si_z\\
=&  \int_{B(e,1)}    (\log_x\circ f)\,d((\ell_z)_*\si_1)\\
=&  \int_{B(e,1)}   (\log\circ \ell_{x^{-1}}\circ f)(y)
(\al\circ\ell_{z^{-1}})(y)\,d\mu(y)\,,
\end{align*}
where  $\mu$ is the left invariant volume form on $G$, $\si_1=\al\,\mu$ and  $x\approx f_1(z)$.   Since $z\approx e$, we get that $d_{CC}(f_1(z),e)<C_1$ by Lemma~\ref{lem_carnot_center_of_mass}.           Also, 
we may assume that $\spt((\ell_z)_*\si_{1})=\spt(\al\circ \ell_{z^{-1} })\subset  B(e,1)$. 
Thus when $y$ is fixed, the integrand is a smooth function in $x$ and $z$ with 
derivatives bounded uniformly independent of $y\in \spt(\al\circ \ell_{z^{-1}})$.  The assertion now follows from the implicit function theorem and the lower bound (\ref{eqn_compression_bound}) on the derivative of $C_{f_*\si_z}$.

(3). For $y\in G$, 
\begin{align*}
f_1(y)=&\com(f_*\si_{y})=\com(f_*(\ell_y)_*\si_1)\\
=&\com((\ell_{f(y)})_*(f_*\si_1))=\ell_{f(y)}(\com(f_*\si_1))\\
=&f(y)\com(f_*\si_1)=f(y)
\end{align*}
since $f_*\si_1$ is symmetric.

(4). If this were false, there would be a compact subset $Y\subset U_1$,
an integer $j$, $\eps>0$, and a subsequence $\{f_{k_i}\}$ such that 
\begin{equation}
\label{eqn_eps_far}
d_{C^j}((f_{k_i})_1\restr_Y,(f_\infty)_1\restr_Y)\geq \eps
\end{equation}
for all $i$.
By (3), all derivatives of $(f_k)_1$ are uniformly bounded on compact subsets of
$U_1$.  Therefore a subsequence converges in $C^j_{loc}(U_1)$, 
to a limit $g_\infty$. To obtain a contradiction with 
\eqref{eqn_eps_far} it only remains to show that
$g_\infty = (f_\infty)_1$ in $U_1$. This identity  follows from the convergence $f_k\stackrel{C^0_{loc}(U)}{\ra} f_\infty$
by passing to the limit in the defining equation for $(f_k)_1(x)$ 
namely
$$ 0 = C_{  (f_k)_* \sigma_x}\big((f_k)_1(x)\big) = \int_{U} (\log \circ \ell_{ ((f_k)_1(x))^{-1}} \circ f_k)(y)   \, \, d\sigma_x(y).$$
To pass to the limit one uses that  for $x \in U_1$ the measure  $\sigma_x = (\ell_x)_* \sigma_1$ has compact support in $U$.

(5).           Let $K \subset U$ be compact. Then there exists a $\rho' > 0$ such that 
$f$ is uniformly continuous in a $\rho'$-neighbourhood of $K$. Let $\omega$ denote a modulus of continuity of $f$ restricted  to that neighbourhood. 
Then for $\rho \in (0, \rho')$ we have $\spt \sigma_{x,\rho} \subset B(x,\rho)$ and $\spt f_* \sigma_{x,\rho} \subset B(f(x), \omega(\rho))$.
Hence by (1) above   and        Lemma~\ref{lem_carnot_center_of_mass}~\eqref{it:lipbound_com}  
 we have
$$
f_\rho(x)= \com f_*(\sigma_{x,\rho})\in B(f(x), C \omega(\rho))
$$ 
for all $x \in K$. Thus  $f_\rho$ converges to $f$, uniformly in $K$. 
\end{proof}

\section{Pansu pullback and mollification}
\label{sec_pansu_pullback_mollification}

The main result of this section is Theorem~\ref{th:main_approximation} below,  which implies Theorem~\ref{th:pull_back} as an easy corollary.

\begin{theorem}[Main approximation theorem]
\label{th:main_approximation}
Let $G$, $G'$, and $f:U\ra U'$ be as in Theorem~\ref{th:pull_back}.
Suppose $\om\in \Om^k(U')$, $\eta\in\Om^{N-k}_c(U) $ are forms with continuous coefficients, such that $\wt(\om)+\wt(\eta)\leq -\nu$.  Then 
\begin{equation}  \label{eq:main_approximation_alternative}
  f_\rho^*\om\wedge \eta \stackrel{L^1_{\loc}}{\lra} f_P^*\om\wedge \eta
\end{equation}
where $f_\rho$ is the mollified map defined in (\ref{eqn_def_f_rho}).
Since $\eta$ has compact support, it follows that 

\begin{equation}  \label{eq:main_approximation}
\int_Uf_P^*\om\wedge \eta=\lim_{\rho\ra 0}\int_Uf_\rho^*\om\wedge \eta\,.
\end{equation}
\end{theorem}

\bigskip\bigskip

\begin{proof}[Proof of Theorem~\ref{th:pull_back} using Theorem~\ref{th:main_approximation}]
If $\om$ is smooth, then 
by Stokes' theorem
\begin{align*}
0&=\lim_{\rho\ra 0}\int_U d(f_\rho^*\om\we\eta)\\
&=\lim_{\rho\ra 0}\left(\int_U f_\rho^*d\om\we\eta+(-1)^k\int_Uf_\rho^*\om\we d\eta\right)\\
&=\int_U(f_P^*d\om\we\eta+(-1)^k\int_Uf_P^*\om\we d\eta\,.
\end{align*}
The general case follows by approximating $\om$ by smooth forms.    
\end{proof}

\bigskip 

We now turn to the proof of the Main Approximation Theorem. 

The strategy of the proof is as follows.  We  show convergence  of the integrals with the help of the Dominated Convergence Theorem.   To verify pointwise convergence $f_\rho^* \omega \wedge \eta\ra f_P^*\om\wedge d\eta$, by simple algebraic manipulations we first express $f_\rho^* \omega \wedge \eta(x)$ in terms of the rescaled, re-centered, and smoothed map $(\de_{\rho^{-1}}\circ f_x\circ \de_\rho)_1$
where $f_x = \ell_{f(x)^{-1}}\circ f\circ \ell_x$.
Then we  use
the definition of the Pansu derivative
and the properties of the center of mass smoothing to see that $(\de_{\rho^{-1}}\circ f_x\circ \de_\rho)_1$ converges in $C^\infty_{\loc}$ to the Pansu differential $D_Pf(x)$ when $x$ is a point of Pansu differentiability.  This gives pointwise convergence of the integrands.

\bigskip
\begin{lemma}
\label{lem_moll_calc}
Let $U \subset G$ be open and   
let $f: U \to G'$ be continuous. 
Suppose $\al\in\Om^k(G')$, $\gamma\in\Om^{N-k}(G)$ are left invariant forms of weight
$w_\al$ and $w_\gamma$,  respectively.
\ben
\item   For every $x\in U_\rho$, 
$$
(f_\rho^*\al  \wedge \gamma)(x)=\rho^{-(\nu + w_\al+w_\gamma)}(h_1^*\al \wedge \gamma)(\de_{\rho^{-1}}(x))\,,
$$
where $h=\de_{\rho^{-1}}\circ f\circ \de_\rho$.
\item  For every $x \in U_\rho$,
$$
(f_\rho^*\al \wedge \gamma)(x)=\rho^{-(\nu + w_\al+w_\gamma)}
\big((\de_{\rho^{-1}}\circ f_x\circ \de_\rho)_1^* \al   \wedge \gamma\big)(e)\,,
$$
where $f_x=\ell_{f(x)^{-1}}\circ f\circ \ell_x$.
\item If $x \in U_\rho$ and  $f(B(x,\rho))\subset B(f(x),C\rho)$, then 
\begin{equation*}
\|(f_\rho^*\al)  \wedge \gamma)(x) \| \leq 
C'\,C^{-w_\al}\rho^{-( \nu + w_\al+w_\gamma)}
\| \alpha\| \, \| \gamma\|.
\end{equation*}
\een
\end{lemma}

\begin{proof}
(1).  Note that  $\{ z : B(z, 1) \subset \delta_{\rho^{-1}} U\} = \delta_{\rho^{-1}} U_\delta$
and thus
$$h: \delta_{\rho^{-1}} U \to G', \quad h_1:  \delta_{\rho^{-1}} U_\delta \to G'.   $$ 
For $x \in U_\rho$ we have
\begin{align*}
(f_\rho^*(\al) \wedge \gamma)(x) 
=& \big(   (\de_\rho \circ h_1 \circ \de_{\rho^{-1}})^* \alpha \wedge \gamma\big)(x)\\
=& (\de_{\rho^{-1}}^* h_1^* \de_\rho^* \alpha \, \wedge\,  \de_{\rho^{-1}}^* \de_\rho^* \gamma)(x)\\
=& \rho^{-(w_\al + w_\gamma)}  (\de_{\rho^{-1}}^* h_1^* \alpha \wedge \de_{\rho^{-1}}^* \gamma)(x)\\
=&  \rho^{-(w_\al + w_\gamma)}  \big(\de_{\rho^{-1}}^* (h_1^* \alpha \wedge  \gamma)  \big)(x)\\
=&  \rho^{-(\nu + w_\al + w_\gamma)} (h_1^* \alpha \wedge  \gamma)(\de_{\rho^{-1}} x)
\end{align*}
In the last step we used that $h_1^* \alpha \wedge  \gamma$ is a multiple of the volume form, which has weight $-\nu$. 

\bigskip

(2). With $h$ as in (1), 
\begin{align*}
h=&\de_{\rho^{-1}}\circ f \circ \de_\rho\\
=&(\de_{\rho^{-1}}\circ \ell_{f(x)}\circ \de_\rho)\circ 
\de_{\rho^{-1}}\circ \ell_{f(x)^{-1}}\circ f\circ \ell_x\circ\de_\rho
\circ(\de_{\rho^{-1}}\circ \ell_{x^{-1}}\circ \de_\rho)\\
=&\ell_{\de_{\rho^{-1}}f(x)}
\circ \de_{\rho^{-1}}\circ f_x\circ \de_\rho
\circ \ell_{\de_{\rho^{-1}}x^{-1}}
\end{align*}
and so
$$
h_1=\ell_{\de_{\rho^{-1}}f(x)}
\circ (\de_{\rho^{-1}}\circ f_x\circ \de_\rho)_1
\circ \ell_{\de_{\rho^{-1}}x^{-1}}\,.
$$
Since $\al$ and $\gamma$ are left invariant we have for $x \in U_\rho$
\begin{align}   \label{eqn_h1_al}
&(    h_1^*\al    \wedge \gamma)
(\de_{\rho^{-1}}(x) ) \notag \\  
=&\ell_{\de_{\rho^{-1}}x^{-1}}^*[(\de_{\rho^{-1}}\circ f_x\circ \de_\rho)_1^*\al
\wedge \gamma] (\de_{\rho^{-1}}(x))    \\
=& [ (\de_{\rho^{-1}}\circ f_x\circ \de_\rho)_1^*\al
\wedge \gamma](e).  \notag
\end{align}
Combining (1) with (\ref{eqn_h1_al}) gives (2).

\bigskip
(3). By our assumptions we have 
$B(\de_{\rho^{-1}}(x),1) \subset \delta_{\rho^{-1}} U$ and     
 $$h(B(\de_{\rho^{-1}}(x),1))
\subset B(h(\de_{\rho^{-1}}(x)),C),$$
   so
$\|D(\de_{C^{-1}}\circ h_1)(\de_{\rho^{-1}}(x))\|\leq C'$
by Lemma \ref{lem_moll_prop}(2).  Using (1) 
and $\de_C^*\al=(\de_{C^{-1}})_*\al=C^{-w_\al}\al$ we get
\begin{align*}
& \|(f_\rho^*\al) \wedge \gamma)(x))\|   \\
=&\|\rho^{-(\nu + w_\al+w_\gamma)}(\de_C\circ(\de_{C^{-1}}\circ h_1))^*\al \wedge \gamma)
(\de_{\rho^{-1}}(x))\|\\
=&C^{-w_\al}\rho^{-(\nu + w_\al+w_\gamma)}\|(\de_{C^{-1}}\circ h_1)^*\al  \wedge \gamma)
(\de_{\rho^{-1}}(x))\|  \\
\leq& C'C^{-w_\al}\rho^{-(\nu + w_\al+w_\gamma)}
\|\al\|    \,  \|\gamma\|.
\end{align*}
\end{proof}

\bigskip
\begin{proof}[Proof of Theorem~\ref{th:main_approximation}]
It suffices to show the assertion if $\omega = a \alpha$ and $\eta = u \gamma$
where $\alpha$ and $\gamma$ are left invariant forms in $\Omega^{k}(G')$ and $\Omega^{N-k}(G)$, respectively, 
with $\wt(\alpha) + \wt(\gamma) \le - \nu$ and where $a$ and $u$ are continuous functions and $u$ has compact support
 in $U$.   Replacing $U$ by an open set $V$ with $\spt u \subset V \subset \subset U$ if necessary we may assume without loss of generality that
\begin{equation}   \label{eq:Lp_integrable}
f \in W^{1,p}(U;G').
\end{equation}

In order to show that 
$$
f_\rho^*\om\wedge \eta \stackrel{L^1_{\loc}}{\lra} f_P^*\om\wedge \eta$$
as $\rho\ra 0$ we will 
use the Dominated Convergence Theorem. 

We first verify pointwise convergence of the integrands
  in   $\spt u \subset U$. Since $\spt u$ is compact there exists a $\rho_0 > 0$ such that  $\spt u \subset U_\rho$
 for all $\rho \in (0, \rho_0)$.
 Moreover it follows from Lemma~\ref{lem_moll_prop}(5)  that $f_\rho \to f$ in $C^0(\spt u)$. Since $f(\spt u)$ is a compact subset of $U'$
 we may assume  that $f_\rho(\spt u) \subset U'$ for  all $\rho \in (0, \rho_0)$
 by making $\rho_0$ smaller if needed.
 In the following we will always assume $\rho \in (0, \rho_0)$. 
Then  the assertions in Lemma \ref{lem_moll_calc} hold for all $x \in \spt u$. Assertion (2) of that lemma gives for $x \in \spt u$
\begin{align*}
&f_\rho^*(\omega)  \wedge \eta (x)\\
=&(a\circ f_\rho)(x)\,u(x) \,    (f_\rho^*\al  \wedge \gamma)(x) \\
=&(a\circ f_\rho)(x) \,u(x)  \, \rho^{-(\nu + w_\al+w_\xi)}
\big( (\de_{\rho^{-1}}\circ f_x\circ \de_\rho)_1^*\al  \wedge \gamma \big)(e)\,.
\end{align*}

If  $f$ is      Pansu  differentiable at $x$,    then 
  $\de_{\rho^{-1}}\circ f_x\circ \de_\rho\stackrel{C^0_{loc}}{\lra}
D_Pf$.  (Recall that we are using the notation $D_Pf(x)$ to denote a graded Lie algebra homomorphism $\fg\ra \fg'$ and a homomorphism of Carnot groups $G\ra G'$, depending on the context.)   So by Lemma \ref{lem_moll_prop} parts (4) and (3) we get
$D(\de_{\rho^{-1}}\circ f_x\circ \de_\rho)_1(e)\ra D_Pf(x)$
as $\rho\ra 0$. 
If $w_\al+w_\gamma= - \nu $, then
\begin{align}  \label{eq:pointwise_convergence_pull_back}
(f_\rho^*(\omega)  \wedge \eta) (x) \ra &    
(a\circ f)(x)\,u(x) \, ((D_Pf(x))^* \al)(x)  \wedge \gamma(x)\\
=&(f_P^*\omega \wedge \eta)(x)  \nonumber
\end{align}
so we have pointwise convergence in this case.  If $w_\al+w_\gamma< - \nu$, then 
$(f_\rho^*(\omega)  \wedge \eta) (x) \ra 0$ as $\rho\ra 0$, while
$$
(f_P^*\omega \wedge \eta)(x)   =(a\circ f)(x)\,u(x)
\, \, ((D_P f)(x)^* \alpha)(x) \wedge \gamma(x).
$$
By 
 Lemma \ref{lem_weight_facts}  parts (3) and (2)     
 the last expression is a form of weight strictly less than $\nu$ and hence vanishes. 
Thus we have 
$\mu$-a.e. pointwise convergence
of the integrand.

Next we find an integrable majorant.  
  Pick $\nu<q<p$.   Let 
  $\psi=g^q$ where $g$ is as in Definition~\ref{def_sobolev_mapping}  in $U$ and extend $\psi$ by zero to $G \setminus U$. 
By    \eqref{eq:Lp_integrable} we have $\psi \in L^{\frac{p}{q}}(G)$. 
Let $M\psi$ be the maximal function:
\begin{equation}
\label{eqn_maximal_function}
M\psi(x)=\sup_{r > 0} \av_{B(x,r)} \psi \, d\mu \,.
\end{equation}
Then it is standard (see \cite[Chapter 2]{heinonen_lectures_analysis_metric_spaces})
 that $M\psi\in L^{\frac{p}{q}}(G)$,  
 and thus $(M\psi)^{-\frac{w_\al}{q}}\in L^1(\spt u)$.  

Let $x \in \spt u$.
If $\rho < \rho_0$ we have
 $B(x, \rho) \subset U$  and using Lemma~\ref{lem_sobolev_embedding}  
we get 
\begin{equation}
\label{eqn_maximal_function_controls_oscillation}
f(B(x,\rho))\subset B(f(x),C_x\rho)
\end{equation}
where 
$$
C_x\lesssim (M\psi)^{\frac{1}{q}}(x)\,.
$$
Now
by Lemma \ref{lem_moll_calc}(3)  we have for all $x \in \spt u$
\begin{align*}
&\| f_\rho^*(\omega ) \wedge \eta (x) \| \\
\leq&\|a\|_\infty   \, \|u\|_\infty \,   \| (f_\rho^*\al \wedge \gamma) (x) \| \\
\leq&C'  C_x^{-w_\al}\rho^{-(\nu + w_\al+w_\gamma)}
  \|a\|_\infty  \, \|u\|_\infty \, \|\al\|   \, \|\gamma\| \\
\lesssim &(M\psi)^\frac{-w_\al}{q}(x)\,.
\end{align*}
Since $f_\rho^*(\omega ) \wedge \eta$ vanishes on $G \setminus \spt u$  this  shows that for $\rho < \rho_0$ the integrands $\|  f_\rho^* \omega \wedge \eta \|$ are  dominated by a fixed $L^1$ function.

\bigskip

\bigskip

\end{proof}

\section{Rigidity of products}
\label{sec_rigidity_products}

In this section we prove Theorems~\ref{thm_main_product_intro}, \ref{thm_qis_between_products_intro} 
and Corollary~\ref{cor_factors_are_qi_intro}, in Subsections~\ref{subsec_proof_product_rigidity}, \ref{subsec_quasiisometries_between_products}, 
and \ref{subsec_factors_are_qi}, respectively.  In  Subsection~\ref{subsec_currents_proof}  we sketch an 
alternative proof of rigidity for quasiconformal homeomorphisms which is
 based
 on the behavior of currents under Sobolev mappings.

\bigskip\bigskip

\subsection{Proof of Theorem~\ref{thm_main_product_intro}}
\label{subsec_proof_product_rigidity}
We first treat a special case which exhibits the essential ideas, and then explain how to modify the argument to handle the general case.

\bigskip
\begin{proof}[Proof of Theorem~\ref{thm_main_product_intro} when $G=G'=\H\times\H$] 

Since $G=G_1\times G_2$ with $G_i\simeq \H$, we have a graded decomposition $\fg=\fg_1\oplus\fg_2$, so the grading $\fg=V_1\oplus V_2$ is of the form $V_i=\oplus_{j}V_{ij}$ where $i,j\in \{1,2\}$ and $V_{ij}:=V_i\cap\fg_j$.  

Let $f:G\supset U\ra U'\subset G$ be as in the statement of Theorem~\ref{thm_main_product_intro}.  We may assume without loss of generality that $U=U_1\times U_2$, where $U_1$, $U_2$ are connected open sets.

\bigskip
{\em Step 1: Any graded automorphism $\phi:\fg\ra\fg$ preserves the decomposition $\fg=\oplus_i\fg_i$, i.e. there is a permutation $\si:\{1,2\}\ra\{1,2\}$ such that $\phi(\fg_i)=\fg_{\si(i)}$.}

Pick $j\in \{1,2\}$.  Then $[V_{1j},V_1]=V_{2j}$, and since $\phi$ is a graded automorphism, it follows that the image $\phi(V_{1j})$ is a $2$-dimensional subspace of $V_1$ such that $[\phi(V_{1j}),V_1]\subset V_2$ is $1$-dimensional.  Hence the projection $\pi_k(\phi(V_{1,j}))\subset V_{1k}$ can be nontrivial for at most one $k\in \{1,2\}$, which means that $\phi(V_{1j})=V_{1\si(j)}$ for some $\si(j)\in \{1,2\}$.  Since $\phi$ is an automorphism the map $\si:\{1,2\}\ra\{1,2\}$ is a bijection.

\bigskip{\em Step 2: There is a permutation $\si$ such that for a.e. $x\in U$, the Pansu derivative $D_Pf(x):\fg\ra\fg$ satisfies $D_Pf(x)(\fg_i)=\fg_{\si(i)}$.}

Let $X_1,\ldots,X_6$ be a graded basis for $\fg$ such that $[X_1,X_2]=-X_3$, $[X_4,X_5]=-X_6$, $V_{11}=\Span\{X_1,X_2\}$, $V_{12}=\Span\{X_4,X_5\}$.  Now let $\th_1,\ldots,\th_6$ be the dual basis, and we adopt the notation from Section~\ref{se:weights}:  $\th_J=\th_{j_1}\wedge\ldots\wedge\th_{j_k}$ if $J=\{j_1<\ldots <j_k\}$. 

Since the Pansu derivative $D_Pf(x)$ is a well-defined measurably varying graded automorphism of $\fg$ for a.e. $x\in U$, by Step 1 we may define a measurable mapping $\si:U\ra S_2$ such that  
\begin{equation}
\label{eqn_derivative_action}
D_Pf(x)(\fg_i)=\fg_{\si(x)(i)}
\end{equation}
for a.e. $x\in U$.  It follows that $f_P^*\th_{123}=a_{123}\th_{123}+a_{456}\th_{456}$ where $a_{123},a_{456}\in L^1_{\loc}$ and
\begin{equation}
\begin{aligned}
\label{eqn_a_123_a_456_si}
a_{456}(x)&=0\;\;\text{when}\;\; \si(x)=\id\\
a_{123}(x)&=0\;\;\text{when}\;\;\si    (x)\neq\id\,.
\end{aligned}
\end{equation}

Pick $\psi\in C^\infty_c(U)$.  Then
$
d(\psi \th_{56})=\sum_{1\leq i\leq 4}(X_i\psi)\th_{i56}
$
so $\wt(d(\psi\th_{56}))\leq -4$.  Since $\wt(\th_{123})=-4$,  by Theorem~\ref{th:pull_back} we have
\begin{align*}
0=\int_Uf_P^*\th_{123}\wedge d(\psi \th_{56})&=\int_U(a_{123}\th_{123}+a_{456}\th_{456})\wedge d(\psi \th_{56})\\
&=\int_Ua_{123}(X_4\psi)\vol_G\,.
\end{align*}
Since this holds for all $\psi$, we get that  $X_4a_{123}=0$ as a distribution.  Similarly $X_5a_{123}=X_1a_{456}=X_2a_{456}$.  This forces $X_6a_{123}=X_3a_{456}=0$ since $X_3=-[X_1,X_2]$ and $X_6=-[X_4,X_5]$.  In view of the fact that $U_1$, $U_2$ are connected, these distributional equations  imply that, up to null sets, $a_{123}$, $a_{456}$ will be nonzero on subsets of the form $A_{123}\times U_2$, and $U_1\times A_{456}$, respectively.  If both subsets had positive measure, then  $a_{123}$, $a_{456}$ would both be nonzero a.e. in $A_{123}\times A_{456}$, which contradicts (\ref{eqn_a_123_a_456_si}).  Therefore, after modifying $\si$ on a set of measure zero, it will be constant in $U$ as desired.

\bigskip
{\em Step 3: $f$ is a product of mappings.}  We may assume without loss of generality that $\si=\id$.

By Fubini's theorem, for a.e. $y\in U_2$, the map $f_y:U\ra G$ defined by $f_y(x)=f(x,y)$ is in $W^{1,p}_{\loc}(U_1)$, and by the chain rule its Pansu derivative satisfies $(D_Pf)(x)(\fg_1)\subset \fg_1$ for a.e. $x\in U_1$.  Hence for a.e. $y\in U_2$ we have $D_P(\pi_2\circ f_y)(x)=0$ for a.e. $x\in U_1$, and since $U_1$ is connected it follows that $\pi_2\circ f_y$ is constant.  Because $p>\nu$ we know that $f$ is continuous, and therefore $\pi_2\circ f_y$ is constant for every $y\in U_2$, i.e.    
   $(\pi_2\circ f)(x,y)$ depends only on $y$.  Similarly $(\pi_1\circ f)(x,y)$ depends only on $x$.  Thus $f$ is a product of mappings.
\end{proof}

\bigskip
We now return to the general case of Theorem~\ref{thm_main_product_intro}.

Let $G:=\prod_{i\in I}G_i$,  $G':=\prod_{j\in I'}G'_j$, and $\fg$, $\fg'$ be the  graded Lie algebras, as in the statement of Theorem~\ref{thm_main_product_intro}.

The generalization of Step 1 of  the above proof is the following result from \cite[Prop. 2.5]{Xie_Pacific2013}.  We include a short proof here.
\begin{lemma}
\label{lem_isom_is_product}
Suppose $\fg=\oplus_{i\in I}\fg_i$, $\fg'=\oplus_{j\in I'}\fg_j'$ where every $\fg_i$, and $\fg'_j$ is nonabelian and does not admit a nontrivial decomposition as a direct sum of graded ideals.
Then any graded isomorphism $\phi:\fg\ra \fg'$ is a product of graded isomorphisms, i.e. there is a bijection $\si:I\ra I'$ and for every $i\in I$ there exists a graded isomorphism $\phi_i:\fg_i\ra \fg'_{\si(i)}$ such that for all $i\in I$ we have $\pi_{\si(i)}\circ\phi=\phi_i\circ\pi_i$.
\end{lemma}
\begin{proof}
For subsets $S_1,S_2$ of $\fg$ or $\fg'$, we let $$[S_1,S_2]=\Span(\{[s_1,s_2]\mid s_i\in S_i\}),$$ 
    and
$\I(S_1)$ be the ideal generated by $S_1$.   Note that the center of $\fg$ intersects $V_1$ trivially, since every $\fg_i$ is nonabelian and has no nontrivial decomposition as a direct sum of graded ideals; similarly the center of $\fg'$ intersects $V_1'$ trivially.

We define a linear subspace $W\subset V_1\subset\fg$
     to be {\bf bracket maximal} if for every $Z\subset V_1$
with $Z\supset W$ and $[Z,V_1]=[W,V_1]$ we have $Z=W$.  Note that
$$
[Z,V_1]=[\oplus_{i\in I}\pi_i(Z),V_1],
$$
so if $W$ is bracket maximal then $W=\oplus_{i\in I}\pi_i(W)$.  A similar definition and remark applies in $\fg'$.

For $i\in I$, let $V_{1i}=V_1\cap \fg_i$.   
Then $V_{1i}$ is bracket maximal for all $i\in I$, since $ V_{1i}\subsetneq Z\subset V_1$ implies that $\pi_j(Z)\neq \{0\}$ for some $j\neq i$, and hence $[\pi_j(Z),V_{1j}]\neq\{0\}$ since the center of $\fg_j$ intersects $V_{1j}$ trivially.

Since $\Phi$ is a graded isomorphism, it follows that $\Phi(V_{1i})\subset V_1'$ is bracket maximal, and $\fg'=\oplus_i \I(\Phi(V_{1i}))$.  But since $\Phi(V_{1i})$ is bracket maximal we have
$$
\Phi(V_{1i})=\oplus_{j\in I'}\pi_{j}(\Phi(V_{1i}))\,,
$$
giving
\begin{align*}
\fg'=\oplus_{i\in I}\I(\Phi(V_{1i}))
=&\oplus_i\oplus_{j}\I(\pi_{j}(\Phi(V_{1i})))\\
=&\oplus_{j}\oplus_{i}\I(\pi_{j}(\Phi(V_{1i})))\,.
\end{align*}
Since $\fg'_j$ does not admit a nontrival decomposition as a direct sum of graded ideals, it follows that for every $j\in I'$,   there is a unique $\tau(j)\in I$  such that  the ideal 
$\I(\pi_{j}(\Phi(V_{1i})))=\{0\}$  for $i\not=\tau(j)$ and  $\I(\pi_{j}(\Phi(V_{1\tau(j)})))=\fg'_j$.  
      Since $\phi$ is a graded isomorphism, $\tau$ is the inverse of a bijection $\si:I\ra I'$, and the lemma follows.
\end{proof}

\bigskip
Let $f:G\supset U\ra G'$    
   be as in the statement of Theorem~\ref{thm_main_product_intro}, where $U=\prod_iU_i$ and the $U_i$s are connected.  By Lemma~\ref{lem_isom_is_product}, we may assume without loss of generality that $I=I'$ and $\fg_i=\fg_i'$ for all $i\in I$, and so there is a measurable function $\si:U\ra \perm(I)$ such that $D_Pf(x)(\fg_i)=\fg_{\si(x)(i)}$ for a.e. $x\in U$. 

For $i\in I$ let $\th_{i}$ denote the pullback $\pi_i^*\vol_{G_i}$ where $\vol_{G_i}$ is a volume form on the factor $G_i$.
Pick $i_0\in I$, and let $I_0\subset I$ be the indices $j\in I$ for which $\fg_j\simeq \fg_{i_0}$.  Then
$$
f_P^*\th_{i_0}=\sum_{j\in I_0}a_j\th_{j}\,,
$$
where $a_j\in L^1_{\loc}(U)$.     First suppose $|I_0| \geq 2$.   For $j,j'\in I_0$  distinct elements  and $X\in V_1\cap \fg_{j'}$,   define a form $\om$ by  $\om:=(\La_{i\neq j,j'}\th_i)\wedge i_X\th_{j'}$.  
    Taking $\psi\in C^\infty_c(U)$, we apply Theorem~\ref{th:pull_back} to get
$$
0=\int_Uf_P^*\th_{i_0}\wedge d(\psi\om)=\pm\int_Ua_j(X\psi)\wedge_{i\in I}\th_i
$$
so $Xa_j=0$ distributionally.  Since $V_1\cap \fg_{j'}$ generates $\fg_{j'}$ as a Lie algebra, we get that $Za_j=0$ for all $Z\in \fg_{j'}$.  Arguing as before, we   get that 
     there is some  $j=\si(i_0)\in I_0$ such that $a_j\neq 0$ and $a_k=0$
 for $k\in I_0\backslash\{j\}$   a.e. in $U$.    If $I_0=\{i_0\}$, then the same assertion holds trivially.
     The index $i_0$ was arbitrary, so we conclude that there is a permuation $\si\in \perm(I)$ such that $D_Pf(x)(\fg_i)=\fg_{\si(i)}$ for a.e. $x\in U$.       The last step of the proof is similar to the special case.
   The proof of   Theorem~\ref{thm_main_product_intro}
      is   now  complete.

\bigskip

\subsection{Proof of Theorem~\ref{thm_qis_between_products_intro}}
\label{subsec_quasiisometries_between_products}
We refer the reader to \cite{bowditch_course_ggt} for background material on quasi-isometries.

Let $\{\hat G_i\}_{i\in I}$, $\{\hat G_j'\}_{j\in I'}$ be collections where each $\hat G_i$, $\hat G'_j$ is either a simply-connected nilpotent Lie group with a left invariant Riemannian metric, or a finitely generated nilpotent group equipped with a word metric.  By \cite{pansu_croisssance} the associated collections of asymptotic cones $\{G_i\}_{i\in I}$, $\{G'_j\}_{j\in I'}$ are Carnot groups equipped with  Carnot-Caratheodory metrics, up to bilipschitz homeomorphism.   We assume that every $G_i$, $G_j'$ is nonabelian and indecomposable as a Carnot group.

Set $\hat G:=\prod_i\hat G_i$, $\hat G'=\prod_jG_j'$, and equip $\hat G$, $\hat G'$ with the $\ell^2$-distance functions 
$$
d^2_{\hat G}(x,x'):={\sum_id_{G_i}^2(\pi_i(x),\pi_i(x'))}\,,\quad
d^2_{\hat G'}(x,x'):={\sum_jd_{G_j'}^2(\pi_i(x),\pi_i(x'))}\,.
$$
If $i\in I$, we say that  $x,x'\in G$ are an {\bf $i$-pair} if $\pi_j(x)=\pi_j(x')$ for every $j\neq i$.

We begin with the following geometric result.

\begin{lemma}
\label{lem_pairs_connected}
Let $X$ be a metric space, where one of the following holds:
\begin{enumerate}[label=(\alph*)]
\item $X$ is isometric to  a simply connected nilpotent Lie group $H$ with a left invariant Riemannian metric, where $H\not\simeq \R$.
\item $X$ is the Cayley graph of a finitely generated nilpotent group $H_0$, where $H_0$ is not virtually cyclic.
\end{enumerate}   
Then there exist $\ul{R}<\infty$, $\la>0$ such that:
\ben
\item   for every $p,x_1,x_2\in X$ such that $R:=d(p,x_1)\geq \ul{R}$, $x_2\in A(p,R,2R)$, there is a path $\ga:[0,1]\ra X\setminus B(x,\la R)$ from $x_1$ to $x_2$.   Here $A(p, R, 2R)=B(p, 2R)\backslash \bar{B}(p, R)$. 
\item Let $P_r:=\{(x,x')\in X\times X\mid d(x,x')\geq r\}$ for $r\in [0,\infty)$.  Then for all $R\geq \ul{R}$, the subset $P_R\subset P_{\la R}$ is contained in a single path component of $P_{\la R}$.
\een
\end{lemma}
\begin{proof}
(1).  Suppose (1) is false.  Then there are sequences $\{p_k\}$, $\{x_{1,k}\}$, $\{x_{2,k}\}$ in $X$ such that $R_k:=d(p_k,x_{1,k})\ra \infty$, $x_{2,k}\in A(p_k,R_k,2R_k)$, but there is no path from $x_{1,k}$ to $x_{2,k}$ in $X\setminus B(p_k,\frac{1}{k}R_k)$.  After passing to a subsequence, the sequence of pointed metric spaces $(X,R_k^{-1}d_X,p_k)$ will pointed Gromov-Hausdorff converge to a pointed limit space $(X_\infty,d_\infty,p_\infty)$ where $(X_\infty,d_\infty)$ is bilipschitz homeomorphic to a Carnot group \cite{pansu_croisssance}.  After passing to a further subsequence, we may assume that the points $x_{1,k},x_{2,k}$ converge to a pair of points  $x_{1,\infty},x_{2,\infty}$ in $X_\infty$, i.e. their images under the same sequence of Gromov-Hausdorff approximations converge to $x_{1,\infty}$ and $x_{2,\infty}$, respectively.   Observe  that $d_\infty(p_\infty, x_{1, \infty})=1$ and 
$1\le d_\infty(p_\infty, x_{2, \infty})\le 2$.  

 Since in case (a) $H\not\simeq \R$ and in case (b) $H_0$ is not virtually cyclic, we know that $X_\infty$ is homeomorphic to $\R^n$ for some $n\geq 2$.  Therefore there is a path $\ga:[0,1]\ra X_\infty\setminus B(p_\infty,3a)$ from $x_{1,\infty}$ to $x_{2,\infty}$ for some $a>0$.  Subdividing $\ga$, we obtain a sequence of points
$$
x_{1,\infty}=z_0,\ldots,z_n=x_{2,\infty}\subset X_\infty\setminus B(p_\infty,3a)
$$
with $d(z_{j-1},z_j)<a$ for all $1\leq j\leq n$.   Hence for large $k$ there exist 
$$
x_{1,k}=z_{0,k},\ldots,z_{n,k}=x_{2,k}\subset X\setminus B(p_k,2aR_k)
$$
such that $d(z_{j-1,k},z_{j,k})<aR_k$ for all $1\leq j\leq n$.  Joining $z_{j-1,k}$ to $z_{j,k}$ by a geodesic $\ga_{j,k}$ and concatenating the $\ga_{j,k}$s, we get a contradiction.

(2).    Let $(x, x'),  (y, y')\in P_R$. We need to show that  $(x, x')$,   $(y, y') $  lie in the same path  component of   $P_{\la R}$.   Iterating (1), we may reduce to the case when $d(x,x')=d(y,y')=10d(x,y)$.  Then applying (1) twice more, we get that $(x,x')$, $(x,y')$ and $(y,y')$ lie in the same path component of $P_{\la R}$.

\end{proof}

\bigskip
Now let $\{\hat G_i\}_{i\in I}$, $\{\hat G'_j\}_{j\in I'}$ be as in Theorem~\ref{thm_qis_between_products_intro}.

\bigskip
Our first step is to show that the image of an $i$-pair
  $(x, x')$    under a quasi-isometry is approximately a $j$-pair for some $j=j(i,x,x')$, provided $d(x,x')$ is large.
\begin{lemma}
\label{lem_i_pair_far_apart}
For every $L\geq 1$, $A<\infty$, $\eps>0$, there is an $R=R(L,A,\eps)<\infty$ such that if $\Phi:\hat G\ra\hat G'$ is an $(L,A)$-quasi-isometry, $i\in I$, and $x,x'\in G$ is an $i$-pair such that $d(x,x')\geq R$, then for every $j\in I'$ we have
$$
\frac{d(\pi_j(\Phi(x)),\pi_j(\Phi(x')))}{d(\Phi(x),\Phi(x'))}\in
[0,\eps)\cup(1-\eps,1]\,.
$$
\end{lemma}
\begin{proof}
If not, then for some $L$, $A$, $\eps$, there is a sequence of $(L,A)$-quasi-isometries $\{\Phi_k:\hat G\ra \hat G'\}$, and for some $i\in I$ a sequence of $i$-pairs $\{x_k,y_k\}\subset \hat G$ such that $d(x_k,x_k')\ra \infty$, and for every $k$ 
\begin{equation}
\label{eqn_not_almost_i_pair}
\frac{d(\pi_j(\Phi_k(x_k)),\pi_j(\Phi_k(x'_k)))}{d(\Phi_k(x_k),\Phi_k(x'_k))}\in
[\eps,1-\eps]\,.
\end{equation}
By pre/postcomposing with translations, we may assume that $x_k=e$ and $\Phi_k(x_k)=e$ for all $k$.  Letting $R_k:=d(x_k,y_k)$, the maps $\Phi_k$ induce $(L,R_k^{-1}A)$-quasi-isometries $(\hat G,R_k^{-1}d_{\hat G})\ra (\hat G',R_k^{-1}d_{\hat G'})$.  Extracting a subsequential limit, we obtain an $L$-bilipschitz homeomorphism $\Phi_\om:G_\om=\prod_iG_i\ra \prod_jG_j'=G_\om'$, and by (\ref{eqn_not_almost_i_pair}) there is an $i$-pair $x_\om,y_\om\in G_\om$ whose image under $\Phi_\om$ is not a $j$-pair for any $j$.  This is a contradiction, since $\{G_i\}_{i\in I}$, $\{G_j'\}_{j\in I'}$ satisfy the assumptions of Theorem~\ref{thm_main_product_intro} and so $\Phi_\om$ is a product mapping.
\end{proof}

\bigskip
We now fix an $(L,A)$-quasi-isometry $\Phi:\hat G\ra \hat G'$.  

For $i\in I$, $R\in [0,\infty)$, let $P_{i,R}$ be the collection of $i$-pairs $x,x'\in G$ with $d(x,x')\geq R$.    Let $R_0:=2L^{-1}A$.  Since $\Phi$ is an $(L,A)$-quasi-isometry, the mapping $\Phi_{ij}:P_{i,R_0}\ra [0,1]$ given by 
$$
\Phi_{ij}(x,x'):=\frac{d(\pi_j(\Phi(x)),\pi_j(\Phi(x')))}{d(\Phi(x),\Phi(x'))}
$$
is well-defined.   Let $R_1=\max(R_0,R(L,A,\frac14))$ where $R(L,A,\frac14)$ is the constant from Lemma~\ref{lem_i_pair_far_apart}.  By that lemma we have a well-defined function
$$
\bar\Phi_{ij}:P_{i,R}\ra \{0,1\}$$
where
$$
\bar\Phi_{ij}(x,x')=
\begin{cases}
1\quad\text{if}\quad \Phi_{ij}(x,x')\in(\frac34,1]\\
0\quad\text{if}\quad \Phi_{ij}(x,x')\in[0,\frac14)\,.
\end{cases}
$$

\bigskip
\begin{lemma}
\label{lem_phi_ij_locally_constant}
There exists a constant  $R_2=R_2(L,A)<\infty$ with the following property.    If 
 $R\geq R_2$ then $\bar\Phi_{ij}(x_1,x'_1)=\bar\Phi_{ij}(x_2,x_2')$ for every  pair $(x_1,x_1'),(x_2,x_2')\in P_{i,R}$ with $d(x_1,x_2),d(x_1',x_2')\leq 1$.
\end{lemma}
\begin{proof}  We have 
$$
d(\Phi(x_1),\Phi(x_2)),d(\Phi(x_1'),\Phi(x_2'))\leq L+A\,,
$$
which implies
\begin{align*}
|\Phi_{ij}(x_1,x_1')-\Phi_{ij}(x_2,x_2')|\leq \frac{2(L+A)}{L(R-1)-2A}
\end{align*}
which is $<\frac14$ when $R>R_2=R_2(L,A)$.  The distance between the intervals $[0,\frac14)$ and $(\frac34,1]$ is $>\frac14$, and therefore $\bar\Phi_{ij}(x_1,x'_1)=\bar\Phi_{ij}(x_2,x_2')$ as claimed. 
\end{proof} 

\bigskip
Combining Lemma~\ref{lem_phi_ij_locally_constant} with Lemma~\ref{lem_pairs_connected}, we get that for $R_3=R_3(L,A)>R_2$, for every $i\in I$, the subset $P_{i,R_3}\subset P_{i,R_2}$ lies in a single path component, and hence $\bar\Phi_{ij}$ is constant on $P_{i,R_3}$.  Applying Lemma~\ref{lem_i_pair_far_apart} to an $i$-pair $(x,x')$ with $d(x,x')$ large, we see that for every $i$ there is a $\si(i)\in I'$ such that  in $P_{i,R_3}$ we have $\bar\Phi_{i\si(i)}\equiv 1$ and $ \bar\Phi_{ij}\equiv 0$ for $j\in I'\setminus\si(i)$.  If $\Phi':\hat G'\ra \hat G$ is an $(L',A')$-quasi-inverse of $\Phi$, where $L'=L'(L,A)$, $A'=A'(L,A)$, then we argue similarly to define  $\Phi'_{ij}$, $\bar\Phi'_{ij}$, $R_3'$, and $\si':I'\ra I$.  Again taking an $i$-pair $(x,x')$ with $d(x,x')$ large, we can find a $\si(i)$-pair $(y,y')\subset P'_{\si(i),R_3'}$ with $y:=\Phi(x)$, $d(y',\Phi(x'))\ll d(\Phi(x),\Phi(x'))$, and applying Lemma~\ref{lem_i_pair_far_apart} we get that $\si'(\si(i))=i$.  Similarly $\si\circ\si'=\id_{I'}$, so $\si$ is a bijection.  Now (\ref{eqn_asymptotically_product}) follows from Lemma~\ref{lem_i_pair_far_apart}.

\bigskip\bigskip
\subsection{Proof of Corollary~\ref{cor_factors_are_qi_intro}}
\label{subsec_factors_are_qi}
Let $\Phi$, $\si$, etc be as in Theorem~\ref{thm_qis_between_products_intro}.  Let $\Phi':\hat G'\ra \hat G$ be a quasi-inverse of $\Phi$, i.e. $\Phi'$ is an $(L,A_1)$-quasi-isometry such that $d(\Phi'\circ\Phi,\id_{\hat G})$, $d(\Phi\circ\Phi',\id_{\hat G'})<A_1$, where $A_1=A_1(L,A)$.   

Choose $i\in I$ and for every $j\neq i$ pick some $x_j\in \hat G_j$.  Now define $\al:\hat G_i\ra \hat G$ by $\pi_j(\al(x))=x_j$ for $j\neq i$, and $\pi_i(\al(x))=x$, and let $\phi:\hat G_i\ra \hat G'_{\si(i)}$ be the composition $\phi:=\pi_{\si(i)}\circ\Phi\circ\al$.  It follows from (\ref{eqn_asymptotically_product}) that $\phi$ is an $(L,A_2)$ quasi-isometric embedding for $A_2=A_2(L,A)$.  We will show that $\phi$ is a quasi-isometry.

Pick $y_0'\in \hat G'_{\si(i)}$.  Now let $y_1\in \hat G_i$ satisfy $d(\phi(y_1),y_0')\leq 2\rho$ where $\rho:=\inf_{x\in \hat G_i}d(\phi(y),y_0')$.  Let $x_0'\in \hat G'$ be the point with $\pi_{\si(i)}(x_0')=y_0'$, $\pi_j(x_0')=\pi_j(\Phi(\alpha(y_1)))$        
  for all $j\in I'\setminus\{\si(i)\}$.  Let $y_2:=\pi_i(\Phi'(x_0'))$.  Now
\begin{align*}
\rho&\leq d(\phi(y_2),y_0')\leq d(\Phi\circ\al(y_2),x_0')\\
&\leq d(\Phi\circ\al(y_2),\Phi(\Phi'(x_0')))+d(\Phi(\Phi'(x_0')),x_0')\\
&\leq Ld(\al(y_2),\Phi'(x_0'))+A+A_1\,.
\end{align*}
Applying Theorem~\ref{thm_qis_between_products_intro} to $\Phi'$, it follows from (\ref{eqn_asymptotically_product}) that 
$$
\rho^{-1}\,d(\Phi'(x_0'),\al(y_2))\leq \eps_1(\rho)\,,
$$
where $\eps_1=\eps_1(L,A)$, and $\eps_1(t)\ra 0$
     as   $t\ra \infty$.
Hence $\rho\leq A_3=A_3(L,A)$.  Thus $\phi$ is an $(L,A_4)$-quasi-isometry.  The index $i\in I$ was arbitrary, so the corollary follows.

\bigskip
\subsection{An alternative approach using currents}
\label{subsec_currents_proof}
We now give a sketch of an alternative approach to rigidity of products.

Let $G:=\prod_{i}G_i$,  $G':=\prod_{j}G'_j$, where $\{G_i\}_{1\leq i\leq n}$, $\{G_j'\}_{1\leq j\leq n'}$ are as in Theorem~\ref{thm_main_product_intro}, and let $\nu$ be the homogeneous dimension of $G$. 

\begin{theorem}
\label{thm_product_rigidity_currents}
Let $F:U\ra U'$ be a quasisymmetric  homeomorphism, where $U\subset G$, $U'\subset G'$ are open subsets.  Then locally $F$ is a product of homeomorphisms.  \end{theorem}

\begin{remark}
We may relax the assumptions on $F$.  For instance, in the proof below it would suffice to assume that:
\bit
\item $F:U\ra U'$ is a $W^{1,p}$-mapping for some $p>\max_i\nu_i$.  Here $\nu_i$ denotes the homogeneous dimension of $G_i$. 
\item $F$ is a homeomorphism.
\item  The (approximate) Pansu differential is an isomorphism almost everywhere.  
\eit 
\end{remark}

\bigskip
We will give only a rough sketch of the argument.  For simplicity, we discuss the special case of a quasisymmetric homeomorphism $F:\H\times\H\ra \H\times \H$.

For $z\in \H$, let $F_z:\H\ra\H\times \H$ be the ``horizontal slice of $F$ at $z$'', i.e. $F_z(x)=F(x,z)$.  Since $F$ is a homeomorphism, it follows that $F_z$ is a proper embedding for every $z$.  By Lemma~\ref{lem_isom_is_product} the Pansu differential $D_PF(x)$ is a product almost everywhere.  Since $F\in W^{1,p}_{\loc}$ for some $p>\nu$, by  Fubini's theorem,  $F_z$ is a  $W^{1,p}_{\loc}$-mapping for a.e. $z\in \H$, and for such $z$, the Pansu differential $D_PF_z(x):\H\ra \H\times\H$ has image equal to one of the two subspaces 
\begin{equation}
\label{eqn_im_horiz_vert}
\fh\oplus\{0\}, \{0\}\times\fh\subset\fh\times\fh
\end{equation}
 for a.e. $x\in \H$. 

The heart of the proof is the following rigidity property of  ``horizontal-vertical'' mappings:
\begin{proposition}[Zig-zag rigidity]
\label{prop_zig_zag_rigidity}
Let $f:\H\ra\H\times\H$ be a continuous proper topological embedding in $W^{1,p}_{\loc}$, $p>4$, and assume that for a.e. $x\in \H$, the image of the Pansu differential $D_Pf(x):\fh\ra\fh\oplus\fh$ satisfies
\begin{equation}
\label{eqn_image_horiz_vert}
\Im(D_Pf(x)) \in \{\fh\oplus\{0\}\,,\,\{0\}\oplus\fh\}\,.
\end{equation}   
Then $\pi_i\circ f:\H\ra\H$ is constant for one of $i\in \{1,2\}$.
\end{proposition}

Applying the proposition with $f:=F_z$, we conclude that $\pi_{i_z}\circ F_z$ is constant for a unique $i_z\in \{1,2\}$.  By continuity, it follows that for every $z\in \H$ and some $i\in\{1,2\}$  independent of $z$, the composition $\pi_{i}\circ F_z$ is constant.  Repeating the argument for vertical slices of $F$ and using the fact that $F$ is a homeomorphism, completes the proof of Theorem~\ref{thm_product_rigidity_currents}.

We now sketch the proof of Proposition~\ref{prop_zig_zag_rigidity}.

Fix an orientation on $\H$, and let $T$ denote the current of integration on $\H$.  This is the de Rham current defined by the formula
$$
\langle T,\om\rangle :=\int_\H\om\,,\quad\forall \om\in\Om^3_c(\H)\,.
$$
Note that $T$ extends continuously to compactly supported forms in $L^1_{\loc}$. 
Since $f$ is a proper mapping, we may use the Pansu pullback to define a pushforward current $T':=(f_P)_\sharp(T)$
$$
\langle T',\om\rangle:=\int_\H f_P^*\om\,.
$$
The condition (\ref{eqn_image_horiz_vert})  implies that $T'$ decomposes as a sum of currents $T'=T'_1+T'_2$, where $T'_i$ has the form $Z_i \mu_i$ where $Z_i$ is the $3$-vector tangent to the $i$-th summand of $\fh\oplus\fh$, and $\mu_i$ is a signed Radon measure.

Using the Pullback Theorem, and an argument ``dual'' to Step 2 of the proof given in Subsection~\ref{subsec_proof_product_rigidity}, it follows that $T'_i$ is closed.  This implies that $\mu_i$ is a product of Haar measure on the $i^{th}$ factor with a Radon measure on the complementary factor.  Consequently, the support  $\spt(T')$ is a union of horizontal and vertical slices.  Letting $S$ be one such slice, since $S\subset\Im f$ and $S$, $\Im f$ are both closed subsets of $\H\times\H$ homeomorphic to $\R^3$, by invariance of domain we have $S=\Im f$.  Hence $\spt(T')=S$, and  the proposition follows.

\section{Complexified        Carnot groups}\label{sec_complex_heisenberg_setup}

\mbox{}
In this section we study mappings between complexified Carnot groups, the complex Heisenberg groups being prime examples.  In Subsection~\ref{subsec_local_results} we discuss local results, in Subsection~\ref{se:global_qc_complex} global rigidity, and in Subsection~\ref{subsec_flexibility_complex_heisenberg_groups} we give examples of mappings exhibiting flexibility.

\bigskip

\bigskip
\subsection{Setup}
We first recall some facts and notation pertaining to complex structures and complexification (see \cite[Chapter IX]{kobayashi_nomizu_vol_II}).

Recall that the complexification $V_\C$ of a real vector space $V$ is the tensor product $V\otimes_\R\C$, and that complexification has obvious compatibilities with tensor/wedge products and complex conjugation:
\begin{align*}
(V\otimes_\R W)_\C&=(V\otimes_\R W)\otimes_\R\C\simeq (V_\C)\otimes_\C(W_\C)\\
\ol{v\otimes w}&=\ol{v}\otimes\ol{w}\,,\qquad v\in V_\C\,,w\in W_\C\,.
\end{align*}
If $M$ is a manifold, then     we    may complexify tensor bundles of $M$ and their spaces of sections, in particular the algebra $\Om^*(M)$ of differential forms, and then pullback and exterior derivative are $\C$-linear and real (i.e. they commute with complex conjugation).  Similarly for Pansu pullback, since it is based on an algebra homomorphism of exterior algebras. 

If $M$ has an almost complex structure $J$, then the complex tangent bundle $T^cM$ decomposes as a direct sum $T^{1,0}M\oplus T^{0,1}M$ of eigenspaces of $J$ corresponding to eigenvalues $i$ and $-i$ respectively.  Similarly, the complex $k$-forms decompose as a direct sum $\oplus_{p+q=k}\Om^{p,q}M$, where $\Om^{p,q}M$ is the space of forms of type $(p,q)$.  If $\al_1,\ldots,\al_n\in \Om^{1,0}M$ is a local basis for the $(1,0)$-forms, then the set of forms $$\al_{j_1}\we\ldots\we\al_{j_p}\we\bar\al_{k_1}\we\ldots\we\bar\al_{k_q}$$ for $1\leq j_1<\ldots<j_p\leq n$, $1\leq k_1<\ldots<k_n\leq n$, form a local basis for $\Om^{p,q}M$.

\bigskip
Now let $H$ be a Carnot group of topological dimension $N$ and homogeneous dimension $\nu$. Let $\fh$ be the  corresponding  Carnot algebra. 
Let $\fg$ denote  the complexified Carnot algebra, i.e. $\fg = \fh^\C$ equipped with the grading $\fg=\oplus_jV_j^\C$.
  The corresponding Carnot group $G$ has topological dimension $2N$ and homogeneous dimension $2\nu$.  We now denote by $J$ the almost complex structure on $G$ coming from $\fg$;  it follows from the Baker-Campbell-Hausdorff formula that $J$ is integrable, i.e. $(G,J)$ is a complex manifold, and the group operations are holomorphic.  Also, complex conjugation $\fg\ra \fg$ is induced by a unique graded automorphism $G\ra G$, since $G$ is simply-connected.

\bigskip

\subsection{Local results}
\label{subsec_local_results}
Retaining the notation above, we now consider a connected open subset $U\subset G$, and a map $f:U\ra G$   in the Sobolev space $W^{1,p}(U,G)$ with $p >2\nu$.

Our main local result is:

\begin{theorem}
\label{thm_qc_complex_heisenberg_holo_antiholo}
Suppose the Pansu differential $D_Pf(x)$ is either a $J$-linear isomorphism or a $J$-antilinear isomorphism for a.e. $x\in U$.  Then $f$ is holomorphic or antiholomorphic (with respect to the complex structure $J$).
\end{theorem}

\begin{corollary}
\label{cor_complexified_holo_antiholo}
Suppose any graded isomorphism $\fg\ra \fg$ is either $J$-linear or $J$-antilinear, and the Pansu differential $D_Pf(x)$ is an isomorphism for a.e. $x\in U$.  Then $f$ is holomorphic or antiholomorphic (with respect to the complex structure $J$).  In particular, any quasisymmetric homeomorphism is $J$-holomorphic or $J$-antiholomorphic.
\end{corollary}
Theorem~\ref{thm_qc_complex_heisenberg_holo_antiholo_intro} follows from Corollary~\ref{cor_complexified_holo_antiholo}, since the complex $m^{th}$ Heisenberg group satisfies the assumptions of Corollary~\ref{cor_complexified_holo_antiholo}  by 
\cite{reimann_ricci}, Section 6 (see also
 Lemma 5.6 in \cite{KMX2}).  Other examples      including $H$ type groups (with odd  dimensional center)   are discussed in Lemma 5.7 in \cite{KMX2}   and the paragraph after that.

The assumption that the Pansu derivative is an isomorphism a.e.\ cannot be dropped. After the end of the proof we give an example
of a Lipschitz map $f:\H^\C\ra\H^\C$  from the complexified first Heisenberg group  to itself for which the Pansu differential is a.e. $J$-linear  or $J$-antilinear,
but which is neither holomorphic nor antiholomorphic.

\begin{corollary}
\label{cor_complex_heisenberg_bilipschitz} Let $f:G\supset U\ra G$ be as in Theorem~\ref{thm_qc_complex_heisenberg_holo_antiholo}, where $U=B(p,r)$ for some $p\in G$, $r>0$.  If $f$ is an $\eta$-quasisymmetric homeomorphism, then:
\bit
\item $f$ is either $J$-biholomorphic or $J$-antibiholomorphic.
\item Modulo post-composition with a Carnot dilation, the restriction of $f$ to the subball $B(p,\frac{r}{K})$ is $K$-bilipschitz, where $K=K(\eta)$.
\eit
\end{corollary}

 We need the following lemma for the proof of Corollary \ref{cor_complex_heisenberg_bilipschitz}.

\begin{lemma}  \label{le:holomorphic_lipschitz}
Let  $G$ be a complexified Carnot group and suppose that $f : G \supset B(p, r) \to B(p', \rho) \subset G$ is  a holomorphic contact map.
Then the restriction of $f$ to the ball $B(p, r/4)$ is $K_G \frac{\rho}{r}$-Lipschitz where $K_G$ depends only on the group $G$.
\end{lemma}

\begin{proof}
Since    Carnot      dilation and left translation are holomorphic  contact maps it suffices to show the result for $p=p'=e$ and $r=\rho=1$.   

Since the topology induced by $d_{CC}$ is the Euclidean topology (see Subsection ~\ref{subsec_carnot_groups}),  the ball $B(e,1)$  is contained in a Euclidean ball $B^{\rm eucl}(e, R_G)$ and by compactness there exists an $r_G > 0$ such that for every point  $ q \in \overline B(e, \frac34)$ 
we have
$B^{\rm eucl}(q, r_G) \subset B(e,1)$. 
Furthermore the  left invariant norm of 
horizontal    vector fields
   on $B(e,1)$ is comparable to 
the Euclidean norm.

 Thus by standard estimates for holomorphic maps $f: B^{\rm eucl}(q, r_G) \to B^{\rm eucl}(e,R_G)$
 we get 
$$  |Df(q)  X|_{\rm eucl}    \le C_N  R_G r_G^{-1}  |X|_{\rm eucl}  \quad \forall q \in B(e, \frac{3}{4})$$
where $N$ is the topological dimension of $G$.
Since the left invariant norm and the Euclidean norm of horizontal   vector fields
  on $B(e, \frac34)$  are equivalent we get
$$  |Df  X|    \le C_G  R_G r_G^{-1}  |X| $$
for all horizontal  vector   fields    
   $X$ in $B(e, \frac34)$.
Finally we use that any two points $q, q'$  in $B(e,\frac14)$ can be connected by a horizontal curve of length $< \frac12$. In particular this
curves stays
inside $B(e,\frac34)$.    
    Thus $d(f(q), f(q')) \le  C_G R_G r_G^{-1} d(q, q')$.
\end{proof}

\bigskip
\begin{proof}[Proof of Corollary~\ref{cor_complex_heisenberg_bilipschitz}]  
After composing $f$ with complex conjugation if necessary, we may assume by Theorem~\ref{thm_qc_complex_heisenberg_holo_antiholo} that $f$ is holomorphic.  

We claim that the Pansu differential of $f^{-1}$ is $J$-linear almost everywhere.  This follows from the fact that quasisymmetric homeomorphisms between open subsets of Carnot groups map sets of measure zero to sets of measure zero \cite[Section 7]{heinonen_koskela}, and that the inverse of a $J$-(anti)linear  automorphism is also $J$-(anti)linear.  Alternatively, since $f$ is a holomorphic homeomorphism, its Jacobian is nonzero away from a complex analytic subvariety $Z\subset U$, so $f(Z)\subset f(U)$ has measure zero.  Hence $f^{-1}$ is locally a holomorphic diffeomorphism on $f(U)\setminus f(Z)$, and so its Pansu differential is $J$-linear almost everywhere.   

Since $f^{-1}$ is $\eta$-quasisymmetric, by Theorem~\ref{thm_qc_complex_heisenberg_holo_antiholo} it is holomorphic, and hence both $f$ and $f^{-1}$ are biholomorphic; furthermore both maps are contact diffeomorphisms because they are Pansu differentiable almost everywhere.

Let $M = \eta(1)$. 
Since $f$ is $\eta$-quasisymmetric we may assume after postcomposing with a dilation that
$B(p', M^{-1} r) \subset  f(B(p,r/2)) \subset B(p',r)$ where $p' = f(p)$. 
Thus by Lemma~\ref{le:holomorphic_lipschitz} the restriction of $f$ to $B(p, r/8)$ is $ 2 K_G$-Lipschitz. 
Let $g$ be the restriction of $f^{-1}$ to $B(p', M^{-1} r)$. Then by Lemma~\ref{le:holomorphic_lipschitz} the restriction of $g$ to $B(p',  M^{-1}  r/4)$ is
$M K_G$-Lipschitz. Moreover $f(B(p, K_G^{-1} M^{-1} r/8))$ is contained in   the ball   $B(p',  M^{-1} r/4)$.  
       Thus the assertion holds with $ K = 8 M K_G $.
\end{proof}

\bigskip

\begin{proof}[Proof of Theorem~\ref{thm_qc_complex_heisenberg_holo_antiholo}]  
The main point in the proof is to show that the pullback of the top degree holomorphic form cannot oscillate between a holomorphic and anti-holomorphic form.   

For simplicity, we begin by proving the result for the complexified first Heisenberg group.   The general case is very similar, although it requires a few minor modifications which we discuss at the end. We first set up some notation. 
Let $\fg$ be the first complex Heisenberg algebra.  Since we will shortly be complexifying a second time,  we recall the convention to
denote the complex multiplication on $\fg$ by $J:\fg\ra \fg$. 
Let $\{X_j,Y_j=JX_j\}_{1\leq j\leq 3}$ be the basis for which $[X_1,X_2]=-X_3$, so
$$[X_1,X_2]=-X_3\,,\quad  [Y_1,X_2]=-Y_3
$$
$$
[X_1,Y_2]=-Y_3\,, \quad [Y_1,Y_2]=J^2(-X_3)=X_3\,,
$$
and all other brackets of basis vectors are zero.  Let $\{\al_j,\be_j\}_{1\leq j\leq 3}$ be the dual basis.   Then 
$$
d\al_3=\al_1\wedge\al_2-\be_1\wedge\be_2\,,\quad d\be_3=\be_1\wedge\al_2+\al_1\wedge\be_2\,.
$$
Now we complexify $\fg$, the exterior algebra, and differential forms.

We define  $(1,0)$ and $(0,1)$-forms
\begin{align*}
\zeta_j=\al_j+i\be_j\,,\quad\bar\zeta_j=\al_j-i\be_j
\end{align*}
and   $(1,0)$ and $(0,1)$-vector fields (the Wirtinger vector fields)   
\begin{align*}
Z_j = \frac12 (X_j -i Y_j)   \quad \bar Z_j = \frac12 (X_j + i Y_j)\,.  
\end{align*}
One gets:
\begin{align*}
\zeta_j(Z_k) &=  \bar \zeta_j(\bar Z_k) = \delta_{jk}, \quad \zeta_j(\bar Z_k) = \bar \zeta_j(Z_k) = 0,
\end{align*}
\begin{align*}
[Z_1, Z_2] &= - \frac12 Z_3, \quad [\bar Z_1, \bar Z_2] = -\frac12 \bar Z_3
\end{align*}
and
\begin{align*}
\al_j=\frac12(\ze_j+\bar\ze_j)\,,\quad \be_j=-\frac{i}{2}(\ze_j-\bar\ze_j)\\
d\ze_3=\ze_1\wedge\ze_2\,,\qquad d\bar\ze_3=\bar\ze_1\wedge\bar\ze_2\,.
\end{align*}
If $u$ is a smooth function
 then we get
\begin{align}
\label{eqn_differential_zs}
du&=\sum_j((X_ju)\al_j+(Y_ju)\be_j)\\
&=\sum_j((Z_ju) \ze_j+(\bar Z_ju) \bar \ze_j)\,.
\end{align}
We use the shortened notation 
$$
\ze_{123}=\ze_1\wedge\ze_2\wedge\ze_3\,,\qquad\ze_{\ol{123}}=\bar\ze_1\wedge\bar\ze_2\wedge\bar\ze_3\,,\text{etc}\,.
$$

The main point is to show that the Pansu differential 
      $D_Pf(x)$    is either $J$-linear for a.e. $x \in U$ or $J$-antilinear for a.e. $x \in U$. 
By assumption for a.e. $x \in U$ the Pansu differential is $J$-linear or $J$-antilinear.      If a graded automorphism $\Phi: \mathfrak g\ra \mathfrak g$  is $J$-linear,   then its complexification $\Phi_\C: \fg_C \ra \fg_C$ carries forms of type $(p,q)$ to forms of type $(p,q)$; in particular we have
$\Phi_\C^* \zeta_j \in \Span_\C\{\zeta_1, \zeta_2, \zeta_3\}$ and hence 
$$ \Phi_\C^*(\zeta_{123}) = a \zeta_{123} $$
where $a  \in \C$. 
If $\Phi$ is $J$-antilinear then
$$  \Phi_\C^*(\zeta_{123}) = a \zeta_{\ol{123}}. $$
Thus if  $f:G\supset U\ra U'\subset G$ is in $W^{1,p}_{\loc}$ for $p>8$ and its Pansu differential is an isomorphism a.e.,  then the complexification 
    $D^{\C}_Pf$   of the Pansu differential  $D_Pf$                 
satisfies 
$$
f_{P,\C}^*\ze_{123}=a_{123}\ze_{123}+a_{\ol{123}}\ze_{\ol{123}}
$$
for $a_{123}$, $a_{\ol{123}}$ measurable.
Note that:
\begin{align*}
d\ze_{123}=d\ze_{\ol{123}}=0\,,\qquad \wt(\ze_{123})=\wt(\ze_{\ol{123}})=-4\,,\\
d\ze_{23}=0\,,\qquad \wt(\ze_{23})=-3\,.
\end{align*}
By linearity the  Pullback Theorem   clearly extends to complex-valued forms if we use the complexification of the Pansu differential.
Applying Theorem~\ref{co:pull_back2} with $\alpha = \zeta_{123}$ and $\beta = \zeta_{23}$
we thus get
\begin{align*}
 0 =  &\int_U  f_{P, \C}^* \zeta_{123}  \wedge d(\varphi \zeta_{23})  = \int_U  a_{\ol{123}} \ze_{\ol{123}} \wedge d\varphi \wedge \zeta_{23} \\
 = & \int_U    a_{\ol{123}}   \, (Z_1 \varphi) \,  \ze_{\ol{123}} \wedge   \ze_{123}.
 \end{align*}
 Hence we get $ Z_1 a_{\ol{123}}=0$ distributionally, and similarly 
$$
 Z_2 a_{\ol{123}}= \bar Z_1a_{123}= \bar Z_2 a_{123}=0\,.
$$

By Lemma~\ref{lem_horizontal_hypoellipticity_1} below $a_{123}$ is holomorphic and $a_{\ol{123}}$ is antiholomorphic.
Now if $a_{123}$ vanishes on a set of positive measure then it vanishes identically since $U$ is connected. 
Since $D_P f(x)$ is an isomorphism for a.e.\ $x$ it follows that $D_P f$ is $J$-antilinear a.e.
On the other hand if $a_{123} \ne 0$ a.e. then $D_P f$ is $J$-linear a.e.

\begin{lemma}
\label{lem_horizontal_hypoellipticity_1}
Suppose $u$ is a $\C$-valued  distribution on $U \subset G$ satisfying
$$
\bar Z_1u=\bar Z_2u=0\,.
$$
Then $u$ is holomorphic. Likewise a  distribution $u$ with $Z_1u=Z_2u=0$ is antiholomorphic.
\end{lemma}

\begin{proof}
 It follows from the definition of the distributional derivative that also $\bar Z_2 \bar Z_1 u = 0$ and $\bar Z_1 \bar Z_2 u = 0$.
In particular $\bar Z_3  u =- 2 [\bar Z_1, \bar Z_2] u = 0$. 
If $\varphi: V \subset U \to \C^3$ is a biholomorhic chart then it follows that $u \circ \varphi^{-1}$
is a distributionally holomorphic distribution defined on an open subset of $\C^3$. 
Thus $u \circ \varphi$ is smooth and  holomorphic and hence $u$ is holomorphic.
\end{proof}

\bigskip
Assume now that $D_P f(x)$ is $J$-linear a.e. We will show that $f$ is holomorphic. 
The main point is to show that for every biholomorphic chart  $\varphi$  the composition $\varphi \circ f$
has vanishing distributional $\bar Z_1$ and $\bar Z_2$ derivatives. Then we can invoke Lemma~\ref{lem_horizontal_hypoellipticity_1}.

For ease of notation we refer to maps which are linear with respect to the multiplication by  $J$ again as $\C$-linear. 
It suffices to show that  for every $x\in U$, there is an open set $V$ containing $f(x)$ and a  biholomorphic chart $\varphi: V \to V'\subset \C^3$ such that the composition 
$$u := \varphi \circ f:U'\ra \C^3$$
is holomorphic,  
   where $U':=f^{-1}(V)$.  After shrinking  $V$ if necessary,  we
may assume that $d \le C d_{CC}$  and that $\varphi$ is bilipschitz when we use the Riemannian metric $d$ on $G$ and the Euclidean metric on $\C^3$. 
Therefore the map $\varphi$ is also Lipschitz as a map from $(V, d_{CC})$ to $\C^3$ with the Euclidean metric. 
Hence $u \in W^{1,p}(U', \C^3)$. We know that $f$ is Pansu differentiable a.e.\ and that  the Pansu derivative is $\C$-linear. 
Moreover it is easy to see that $\varphi$ viewed as a map from a subset of the Carnot group $G$ to the abelian group $\C^3$ 
is Pansu differentiable and 
   $D_P \varphi = (D \varphi|_{V_1})\circ   \, \pi_1$,   
     where $\pi_1$ is the projection form $\fg$ to its first layer $V_1$. 
In particular $D_P \varphi$ is $\C$-linear. By the pointwise chain rule for the Pansu derivative we see that $u$ is Pansu 
differentiable a.e.\  and 
\begin{equation}  \label{eq:pointwise_chain_holo}
 D_P u(x) = D_P \varphi(f(x))    D_P f(x).
 \end{equation}
Thus $D_P u(x)$ is $\C$-linear a.e. and in particular $D_P u(x) \bar Z_i = 0$ for $i=1,    2$.  
By Lemma~\ref{le:pansu_vs_distributional}  below this implies that $\bar Z_i u = 0$ in the sense of distributions and it then follows from 
Lemma~\ref{lem_horizontal_hypoellipticity_1} that $u$ is holomorphic. This finishes the proof for the complexified first Heisenberg group.

\bigskip
\begin{lemma} \label{le:pansu_vs_distributional} Assume that $G$ is a  Carnot group with homogeneous dimension $ \nu$. 
Assume that  $u: U \subset G \to \C$ is   in  the Sobolev space $W^{1,p}(U,\C)$ with $p >  \nu$. Then for every horizontal 
left invariant vectorfield $X$ the distributional derivative $X u$ agrees with $D_P u X$, i.e.
$$ \int_G D_P u(x) X(x)  \, \,  \eta(x) \, d\mu(x) = - \int_G u   \, (X \eta) \, d\mu$$    for any $\eta \in C_c^\infty(U)$.   
\end{lemma}
\begin{proof}
Let $\eta \in C_c^\infty(U)$ and define
$$ \omega(t) = \int_G  u(x \exp t X) \eta(x)   \, d\mu(x).$$
Since $u$ is a.e. Pansu differentiable by Theorem~\ref{thm_pansu_differentiability}, we have
$$
t^{-1} [u(x \exp tX) - u(x)] \to D_P u(x) X
$$ as $t \to 0$ for a.e. $x \in U$. 
Moreover by the Sobolev embedding theorem (see Lemma~\ref{lem_sobolev_embedding})  this difference quotient
is bounded by the local  maximal function of the horizontal derivative and hence by 
a fixed $L^p$ function (cf. (\ref{eqn_maximal_function}) and (\ref{eqn_maximal_function_controls_oscillation})). Thus the dominated convergence theorem implies that
$$ \omega'(0) = \int_G D_P u(x) X(x)  \, \,  \eta(x) \, d\mu(x).$$

Using the change of variables $y = r_{\exp tX} x = x \exp tX$ and the bi-invariance of $\mu$  we get
$$ \omega(t) = \int_G  u(y) \eta(y \exp(-tX)) \, d\mu(y).$$
Differentiation with respect to $t$ shows that
$$ \omega'(0) =  - \int_G u   \, (X \eta) \, d\mu.$$
Comparing the two expressions for $\omega'(0)$ we get the desired identity.
\end{proof}

\bigskip\bigskip
We finally indicate the necessary modifications for a general complexified Carnot group $G$ with Carnot algebra $\fg$. 
As before we denote the complex structure on $\fg$ by $J$. 
Now we again complexify $\fg$. 

It is easy to see that the complexified algebra $\fg^\C$ is the direct sum of 
the eigenspaces of (the complexification of) $J$ :
$$ \fg^\C = \fg^\C_i \oplus \fg^\C_{-i}$$
where
$$ \fg^\C_{\pm i} = \{ X \mp i JX :  X \in \fg\}\,;$$
moreover  
\begin{equation} \label{eq:J_eigenspaces_commute}
[\fg^\C_i, \fg^C_i] \subset \fg_i^\C, \quad [\fg^\C_{-i}, \fg^\C_{-i}] \subset \fg^\C_{-i}, \quad [\fg^\C_i, \fg^\C_{-i}] = \{0\}\,
\end{equation}
and conjugation exchanges the subalgebras $\fg^\C_{\pm i}$.  See \cite{KMX2} for the details.

For each of the subalgebras $\fg^\C_{\pm i}$ we consider the space $\Lambda^k \fg^\C_{\pm i}$ of  complex-valued $k$-forms. A complex-valued
 one-form $\alpha \in \Lambda^1 \fg^\C_i$ has natural extension to
a complex one-form on $\fg^\C$ by setting it zero on $\fg^\C_{-i}$ and $k$-forms are extended similarly. We denote the spaces of the extended
forms also by $\Lambda^k \fg^\C_{\pm i}$ and view them as a subspaces of $\Lambda^k \fg^\C$. 
Pullback by  (the complexification of) a $J$-linear graded isomorphism preserves the spaces $\Lambda^k \fg^\C_{\pm i}$
while pullback by $J$-antilinear isomorphisms permutes them. 
It follows from the fact that 
$ [\fg^\C_i, \fg^\C_{-i}] = \{0\}$ and formula  \eqref{eq:exterior_derivative_on_algebra}
 that the exterior derivative maps  $\Lambda^k \fg^\C_{i}$ to 
 $\Lambda^{k+1} \fg^\C_{i}$ and $\Lambda^k \fg^\C_{-i}$ to 
 $\Lambda^{k+1} \fg^\C_{-i}$.

Now let $\omega$ be a non-zero top degree from in $\Lambda^{N} \fg^\C_i$. Note that  $\omega$ is determined uniquely up to a complex factor.
Let $\bar \omega \in \Lambda^{N} \fg^\C_{-i}$ denote the pullback by complex conjugation. Then  the complexification $\Phi_\C$  of  a $J$-linear 
graded isomorphism $\Phi: \fg \to \fg$ satisfies $\Phi_\C^* \omega = a \omega$ with $a \in \C$. 
Similarly for the complexification of a $J$-antilinear graded isomorphism we get $\Phi_{\C}^* \omega = b \bar \omega$ with $b \in \C$. 
Thus by assumption we have
$$ f_{P, \C}^* \omega = a \omega + b \bar \omega$$
with measurable functions $a$ and $b$. 

We will show again that $a$ is holomorphic and $b$ is antiholomorphic.
Let $X$ be in the first layer of  $\fg^{\C}_i$. Consider $\omega$ as an $N$-form on the subalgebra  $\fg^\C_i$.  Then it follows from Cartan's formula and the  bi-invariance of $\om$ that the $N-1$ form $i_X \omega$ is closed. 
 As discussed above it follows from the identity $[\fg^\C_i, \fg^\C_{-i}] = \{0\}$ that  the extension of $i_X \omega$ to an $N-1$ form on $\fg^\C$ is also
closed. We apply the pullback theorem with $\alpha = \omega$
 and let $\beta = \ i_X \omega$. Thus we get for every $\varphi \in C_c^\infty(U)$
 \begin{align*} 0 &= \int_U f^*_{P, \C} \alpha \wedge d(\varphi \beta) = \int_U  b \bar \omega  \wedge  d\varphi \wedge i_X \omega\\
 &= \int_U b  (X \varphi) \, \bar \omega \wedge \omega
 \end{align*}
 where we have used $\alpha \wedge i_X \omega = i_X \alpha \wedge \omega = \alpha(X) \omega$ 
when $\omega$ is a top-degree form,  $\alpha \in \Lambda^1 V$, $X \in \Lambda_1 V$. 
Since $\bar \omega \wedge \omega$ is (a fixed multiple of) the volume form on $G$ we deduce that
$X b = 0$ in the sense of distributions for every horizontal left invariant vectorfield  in $\fg_i^\C$. By taking commutators
it follows that $X b = 0$ for every left invariant vectorfield in $\fg_i^\C$. As in  the proof of 
Lemma~\ref{lem_horizontal_hypoellipticity_1} we see that 
$b$ is distributionally antiholomorphic and hence antiholomorphic. Similarly $a$ is holomorphic.
Since $D_P f(x)$ is an isomorphism for a.e. $x$ we deduce as before that $a \ne 0$ a.e.\ or $b \ne 0$ a.e.
In the first case $D_P f(x)$ is $J$-linear for a.e.\ $x$, in the second case $D_P f(x)$ is $J$-antilinear for a.e.\ $x$.
The argument that this implies that $f$ is holomorphic or antiholomorphic is the same as for the complexified first Heisenberg group.
\end{proof}

\bigskip

We finally show that the conclusion of  
Theorem~\ref{thm_qc_complex_heisenberg_holo_antiholo} does not hold if the assumption that the Pansu derivative is a.e.\ invertible is dropped.
Identify the first complexified Heisenberg group  $\H^\C$ with $\C^3$ via the exponential map.
Then the standard graded basis of left invariant  holomorphic and antiholomorphic one-forms is given by
$$ \zeta_1 = dz^1, \quad \zeta_2 = dz^2, \quad \zeta_3 = d z^3 + \frac12 (z^1 dz^2 - z^2  dz^1) $$
and $\overline \zeta_1 = d\overline{z^1}$ etc. 
Thus the  map $\gamma : \C \to \H^\C$ given by $\gamma(z) =(z,0,0)$ is holomorphic and an isometry. 
The map $\pi_1: \H^\C \to \C$ given by $\pi_1(z_1, z_2, z_3) = z_1$ is  holomorphic and has Lipschitz constant 1.
Let $g: \C \to \C$ be a Lipschitz map such that for a.e.\  $z \in \C$ the matrix  $Dg(z)$ is conformal or anticonformal.
Then $f = \gamma \circ g \circ \pi_1$ is a Lipschitz map from $H^\C$ to itself. In particular $f$ is in the Sobolev space
$W^{1,p}_{\rm loc}(\H^\C, \H^\C)$ for all $p \le \infty$. Moreover for a.e. $x \in \H^\C$  the Pansu differential $D_P f(x)$ is holomorphic or antiholomorphic.
Now if we let $g$ be the folding map $g(x +i y) = |x| + iy$ then $f$ is holomorphic for $z_1 > 0$ and antiholomorphic for $z_1 < 0$.

\bigskip\bigskip

\bigskip

\subsection{Global quasiconformal homeomorphisms of complexified Carnot groups}
\label{se:global_qc_complex}
We now study global quasiconformal homeomorphisms of complexified Carnot groups.
For simplicity we first state the result for the first complex Heisenberg group $\H^\C$.

\begin{theorem}
\label{lem_global_qc_complex_heisenberg}

\mbox{}
\ben
\item Any $K$-quasiconformal homeomorphism $\phi:\H^\C\ra\H^\C$ is $\eta$-quasisymmetric, where $\eta=\eta(K)$.
\item Any $K$-quasiconformal homeomorphism $\phi:\H^\C\ra\H^\C$ is $L$-bilipschitz for $L=L(K)$, modulo precomposition by a suitable Carnot rescaling.  In other words, $\phi$ is an $L$-quasisimilarity, where $L=L(K)$.
\item If $u:\H^\C\ra \C$ is a Lipschitz holomorphic function, then $u=\bar u\circ \pi$, where $\pi:\H^\C\ra \C^2$ is the abelianization homomorphism, and $\bar u:\C^2\ra\C$ is an affine holomorphic function.
\item Any holomorphic qc homeomorphism $\H^\C\ra \H^\C$ is a composition of a left translation and a complex graded automorphism.
\item The group of qc homeomorphisms $\H^\C\ra\H^\C$ is generated by left translations, complex graded automorphisms, and complex conjugation.
\een
\end{theorem}

\bigskip
In the proof we will need a lemma characterizing affine maps.
\begin{definition}
Let $\Phi: \fg\ra \fg$ be a linear map.  A $C^1$ map $f:G\ra G$ is a {\bf $\Phi$-map} if $Df(p)(X)=\Phi(X)(f(p))$ for all $X\in \fg$, $p\in G$.  Here we are identifying tangent spaces with the Lie algebra via left translation.
\end{definition}

\begin{lemma}
\label{lem_affine}
Let  $f:G\ra G$ be a $\Phi$-map.  Then $f$ is affine, i.e. a composition of a left translations with a homomorphism.  In particular $\Phi$ is a Lie algebra homomorphism.
\end{lemma}

\begin{proof}
Note that $\Phi$-maps are preserved by left translation:  if $f$ is a $\Phi$-map, then so is $\ell_{g_1}\circ f\circ \ell_{g_2}$ for all $g_1, g_2\in G$.  Also, if two $\Phi$-maps $f_1$, $f_2$ agree at some point, then they coincide.  To see this, let $W=\{g\in G\mid f_1(g)=f_2(g)\}$.   Then $W$ is closed, and if $g$ lies in  $W$, then so does the integral curve $t\mapsto g\exp(tX)$ of the left invariant vector field $X$, by uniqueness of integral curves of the smooth vector field $\Phi(X)$.   Hence $W$ is also open, and therefore $W=G$ by the connectedness of $G$.

Suppose $f$ is an $\Phi$-map with $f(e)=e$.  Then for every $g\in G$
$$
(f\circ\ell_g)(e)=f(g)=f(g)\cdot f(e)=(\ell_{f(g)}\circ f)(e)\,,
$$
so $f\circ\ell_g\equiv \ell_{f(g)}\circ f$.   Hence for all $g_1,g_2\in G$
$$
f(g_1g_2)=(f\circ \ell_{g_1})(g_2)=(\ell_{f(g_1)}\circ f)(g_2)=f(g_1)f(g_2)
$$
so $f$ is a homomorphism of Lie groups, and $\Phi$ is its derivative. 
\end{proof}

\bigskip\bigskip
\begin{proof}[Proof of Theorem~\ref{lem_global_qc_complex_heisenberg}]
(1).  This is a general fact about Loewner spaces (or spaces satisfying Poincare inequalities), see \cite{heinonen_koskela}.   

(2).  We may assume that $\phi$ is holomorphic.  Pick $p\in \H^\C$, $R\in (0,\infty)$.  By (1), we may assume after postcomposing $\phi$ with a suitable Carnot rescaling that $B(\phi(p),R/C)\subset \phi(B(p,R))\subset B(\phi(p),CR)$ for $C=C(K)$.  Now the derivatives of $\phi$ in $B(p,R/2)$ are uniformly bounded, since $\phi$ is holomorphic.  Similarly for $\phi^{-1}$.  This implies that modulo rescaling, the horizontal derivatives are uniformly bilipschitz, and $\phi$ itself is uniformly bilipschitz.

(3). $Z_ju:\H^\C\ra\C$ is holomorphic for $1\leq j\leq 3$.  Since $\|Z_j\|$ is bounded for $1\leq j\leq 2$, so is $Z_ju$.  But $\H^\C$ is biholomorphic to $\C^3$ by the exponential map, so $Z_ju$ is constant for $1\leq j\leq 2$.  Hence $Z_3u=[Z_1,Z_2]u=0$. But $\bar Z_3u=0$ since $u$ is holomorphic.  Therefore $X_3u=Y_3u=0$, and $u$ is constant along cosets of the center.  Therefore $u=\bar u\circ \pi$, where $\pi:\H^\C\ra \C^2$ is the abelianization.  Also, $\bar u$ is Lipschitz and holomorphic, so it is an affine holomorphic function.

(4).  Let $\phi:\H^\C\ra \H^\C$ be a holomorphic quasiconformal homeomorphism. By (3), it follows that $\phi$ preserves cosets of the center, and descends to  a holomorphic affine transformation $\bar \phi:\C^2\ra\C^2$.  Modulo composition with a left translation and a complex graded automorphism, we may therefore assume that $\phi(e)=e$ and that $\bar\phi=\id_{\C^2}$.  But now the horizontal derivative of $\phi$ preserves horizontal left invariant vector fields, so by Lemma~\ref{lem_affine} the map $\phi$ is a left translation, and since $\phi(e)=e$ we have $\phi=\id_{\H^\C}$.

(5). We already know that quasiconformal homeomorphisms are either holomorphic or antiholomorphic, so we are done by (4).

\end{proof}

Returning to the general case of complexified Carnot groups, we let $G$ and $H$ be as in Subsection~\ref{subsec_local_results}.  Note that by Corollary~\ref{cor_complex_heisenberg_bilipschitz}, if $f:G\ra G$ is a quasiconformal homeomorphism whose Pansu differential is $J$-linear or $J$-antilinear almost everywhere, then $f$ is biholomorphic or $J$-antibiholomorphic; therefore the collection of such homeomorphisms is just the group of quasiconformal $J$-(anti)biholomorphic mappings of $G$.

\begin{theorem}
The group of quasiconformal $J$-(anti)biholomorphic mappings $G\ra G$ is generated by left translations, complex graded automorphisms, and complex conjugation.
\end{theorem}
\begin{proof}
We proceed as in   Theorem~\ref{lem_global_qc_complex_heisenberg}, except that we replace $\H^\C$ with $G$, and require the Pansu differentials of quasiconformal homeomorphisms to be almost everywhere $J$-linear or $J$-antilinear.
 
The proofs of (1)-(3) and (5) follow almost verbatim.  

(4). Consider a horizontal left invariant vector field $X$.  The pushforward $\phi_*X$ is a horizontal vector field; we just want to see that it is a left invariant vector field.  To that end, consider a $\C$-linear function $v:G/[G,G]\simeq\C^n\ra\C$, and observe that $u:=   v\circ\pi\circ  \phi$    
   is a Lipschitz holomorphic function, so (3) applies.  This gives that $\phi_*X$ is left invariant.  Taking brackets, it follows that $\phi_*X$ is left invariant for any left invariant vector field $X$.  Hence $\phi$ is affine by Lemma~\ref{lem_affine}.  Modulo composition with a left translation, it is a holomorphic contact automorphism, and hence graded.
\end{proof}

\bigskip
\subsection{Flexibility of the complex Heisenberg groups}
\label{subsec_flexibility_complex_heisenberg_groups}
We exhibit an abundance of global contact diffeomorphisms $\H_n^\C\ra\H_n^\C$, and in particular, an abundance of local contact diffeomorphisms.  Note that the local assertion follows from the work of Ottazzi-Warhurst, since the complex Heisenberg groups are nonrigid in the sense of \cite{ottazzi_warhurst}.

We let $X_1,\ldots,X_{2n+1}$ be a graded basis for the Lie algebra $\fh_n$ where $X_1,\ldots,X_{2n}\in V_1$, $X_{2n+1}\in V_2$, and $[X_j,X_{j+n}]=-X_{2n+1}$.  We denote the dual left invariant coframe by $\th_1,\ldots,\th_{2n+1}\in\Om^1(\H_n)$, so  $d\th_{2n+1}=\sum_{j=1}^n\th_j\wedge\th_{j+n}$.  Then $d\th_{2n+1}$ descends to a left invariant $2$-form $\ol{d\th_{2n+1}}$ on the abelianization $\H_n/[\H_n,\H_n]$; we may identify the pair $(\H_n/[\H_n,\H_n],\ol{d\th_{2n+1}})$ with $(\R^{2n},\om_n)$, where  $\om_n=\sum_jdx_j\we dx_{j+n}\in \Om^2(\R^{2n})$ is the standard symplectic form on $\R^{2n}$.

Now consider the complexification of $\H_n$, i.e. the complex Heisenberg group $\H_n^\C$, whose  Lie algebra we identify with the complexification $\fh_n^\C$.  Then $\{X_j,iX_j\}_{1\leq j\leq 2n+1}$ is a graded basis for $\fh_n^\C$ over $\R$, and we let $\{\al_j,\be_j\}_{1\leq j\leq 2n+1}$ be the dual basis.  Now $\{\zeta_j:=\al_j+i\be_j\}_{1\leq j\leq 2n+1}$ is the corresponding basis of left invariant $(1,0)$-forms, and $d\zeta_{2n+1}=\sum_{j=1}^n\zeta_j\we\zeta_{j+n}$.  The form $d\zeta_{2n+1}$ descends to a form $\widehat{d\zeta_{2n+1}}$ on the abelianization    
and we may identify $(\H_n^\C/[\H_n^\C,\H_n^\C],\widehat{\zeta_{2n+1}})$ with $(\C^{2n},\om_n^\C)$ where  $\om_n^\C=\sum_{j=1}^ndz_j\we dz_{j+n}\in \Om^2(\C^{2n})\otimes\C$ is the standard holomorphic symplectic form on $\C^{2n}$.  

\begin{lemma}
Pick $n\geq 1$.  Let $\pi:\H_n\ra \R^{2n}$, $\pi:\H_n^\C\ra \C^{2n}$ denote the abelianization homomorphisms.
\ben
\item If $\phi:\R^{2n}\ra\R^{2n}$ is a symplectic diffeomorphism, i.e. $f^*\om_n=\om_n$,  then there is a contact diffeomorphism $\hat\phi:\H_n\ra \H_n$ lifting $\phi$, i.e. $\pi\circ\hat\phi=\phi\circ\pi$.  Morever $\hat\phi$ is unique up to post-composition with central translation.
\item If $\phi:\C^{2n}\ra\C^{2n}$ is a diffeomorphism which preserves $\om_n^\C$, then there is a contact diffeomorphism $\hat\phi:\H_n^\C\ra \H_n^\C$ lifting $\phi$, i.e. $\pi\circ\hat\phi=\phi\circ\pi$.  Morever $\hat\phi$ is unique up to post-composition with central translation.
\een
\end{lemma}
\begin{proof}
The homomorphisms $\pi:\H_n\ra \R^{2n}$, $\pi:\H_n^\C\ra \C^{2n}$ define principal bundles with structure groups $\R$ and $\C$ respectively.  Also,  $\th_{2n+1}$, $\th_{2n+1}^\C$ are connection $1$-forms and $\om_n$, $\om_n^\C$ are the curvature forms on the two bundles, respectively.  Hence (1) and (2) follow from Lemma~\ref{lem_principal_bundle_abelian}.
\end{proof} 
We recall a standard fact about connections on principal bundles with abelian structure group:
\bigskip
\begin{lemma}
\label{lem_principal_bundle_abelian}
Pick $k\geq 1$.  For $j\in \{1,2\}$ let $\pi_j:P_j\ra B_j$ be a principal $\R^k$-bundle with connection form $\th_j\in \Om^1(P_j;\R^k)$ and curvature form $\om_j\in \Om^2(B_j;\R^k)$.   Assume that $H^1(B_1,\R)=\{0\}$.

If $\phi:B_1\ra B_2$ is a smooth map such that $\phi^*\om_2=\om_1$, then there is a lift $\hat\phi:P_1\ra P_2$ to a connection preserving principal bundle mapping; moreover if $B_1$ is connected then $\hat\phi$ is unique up to (composition with) the action of the structure group $P_1\curvearrowleft\R^k$. 
\end{lemma}
\begin{proof}
By replacing $\pi_2:P_2\ra B_2$ with the pullback bundle $\phi^*\pi_2:  \phi^*P_2\ra B_1$ and $\th_2$ with the pullback connection $\phi^*\th_2$, we may assume that $B_2=B_1$ and $\phi=\id_{B_1}$. There is a principal bundle equivalence $\psi:P_1\ra P_2$ over $\id_{B_1}$  because $\R^k$ is contractible.  Since $\psi^*\th_2$ and $\th_1$ are both connection forms for $\pi_1:P_1\ra B_1$, the difference $\psi^*\th_2-\th_1$  descends to a $1$-form $\bar\th\in \Om^1(B_1;\R^k)$, where $d\bar\th=\om_2-\om_1=0$.  Given a smooth map $u:B_1\ra \R^k$, we may compose $\psi$ with the corresponding shear, i.e. $\psi'(p)=\psi(p)\cdot u(\pi_1(p))$.  Then $(\psi')^*\th_2-\th_1$ descends to $\bar\th'=\bar\th+du$.  Using the fact that $H^1(B_1,\R)=\{0\}$, the closed form $-\bar\th$ is exact, i.e. we may choose $u$ such that $\bar\th'=0$, and $u$ is unique up to a constant function since $B_1$ is connected.
\end{proof}

\bigskip\bigskip
We may readily produce symplectic shears as follows.

Let $u:\R^{2n}\ra\R$ be a smooth function which depends only on the first $n$-coordinates, i.e. $\D_ju=0$ for $n+1\leq j\leq 2n$.  Let $X_u$ be the Hamiltonian vector field of $u$, i.e. $\om_n(X_u,\cdot)=du(\cdot)$.  Then $X_u$ is tangent to the vertical $n$-planes and invariant under every translation $x\mapsto x+v$ for $v\in \{0\}\times\R^n$, and is therefore complete.  The time $1$ map $\Phi^1:\R^{2n}\ra\R^{2n}$ of the Hamiltonian flow is a symplectic diffeomorphism.

Similarly, if $u:\C^{2n}\ra\C^{2n}$ is a holomorphic function which depends only on the first $n$-coordinates, then the time $1$ map of the associated holomorphic Hamiltonian vector field is a symplectic biholomorphic mapping with respect to the holomorphic symplectic form $\om_n^\C=\sum_{j=1}^ndz_j\we dz_{j+n}$.

\bigskip\bigskip
\appendix

\section{Equiregular manifolds}
\label{sec_equiregular_manifolds}~

In this appendix we will discuss  versions of the Approximation Theorem and Pullback Theorems for Sobolev mappings between equiregular subriemannian manifolds.  We will use (a variation on) the setup from Gromov's paper, see \cite[Section 1.4]{gromov_carnot_caratheodory_space_seen_from_within}, \cite{vodopyanov_carnot_manifolds}. We refer the reader to these papers for more detail.

We will be using subriemannian manifolds with an adapted frame.  These are tuples $(M,V_1,g,\{X_1,\ldots,X_N\})$ where:
\bit
\item $M$ is a smooth $N$-manifold.
\item $V_1\subset TM$ is an equiregular subbundle, so we have a filtration of $TM$ by smooth subbundles
$$
V_1\subsetneq\ldots\subsetneq V_s=TM
$$
where $V_j$ is spanned by brackets of smooth sections of $V_1$ and $V_{j-1}$ for $j\geq 2$, and $s \geq 1$.
\item  $g$ is a smooth Riemannian metric.
\item $X_1,\ldots,X_N$ is a smooth $g$-orthonormal frame.
\item There is a ``degree'' function 
$$
\deg:\{1,\ldots,N\}\ra \{1,\ldots, s\}
$$ 
such that if $W_j$ is the subbundle spanned by $\{X_i\mid \deg i=j\}$, then  $V_j:=\oplus^\perp_{i\leq j}W_i=\Span\{X_i\mid \deg i\leq j\}$.  
\eit

\bigskip
In addition we will use the following objects which are associated with the above structure:
\bit 
\item The coframe $\th_1,\ldots,\th_N$ dual to $X_1,\ldots,X_N$.
\item For $r\in (0,\infty)$, the anisotropically rescaled frame $\{X_j^r\}$ where $X_j^r:=r^{\deg j}X_j$.
\item For every $x\in M$, the graded Lie algebra $\fg_x:=\oplus_j(V_j)_x/(V_{j-1})_x$ with bracket induced by the Lie bracket.  We identify $\fg_x$ with $T_xM$ using the isomorphism 
\begin{equation}
(W_j)_x\simeq (V_j)_x/(V_{j-1})_x
\end{equation}
induced by the inclusion $(W_j)_x\ra (V_j)_x$.
\item The associated Carnot-Caratheodory metric $d_{CC}$, with Hausdorff dimension $\nu=\sum_{j=1}^sj\dim W_j$.
\item The (non-Riemannian) connection $\nabla$ on $TM$ for which the $X_j$s are parallel vector fields.
\item The exponential map $\exp:TM\supset U\ra M$ associated with $\nabla$, which sends $v\in TM$ to the $\ga_v(1)$, where $\ga_v$ is the $\nabla$-geodesic with initial velocity $v$.
\eit

A key ingredient will be Gromov's blow-up theorem for subriemannian geometries, which gives a quantitative comparison between the subriemannian structure at a small scale near $x\in M$, and the nilpotent tangent cone $\fg_x$.   
\begin{theorem}[\protect{\cite[Section 1.4]{gromov_carnot_caratheodory_space_seen_from_within}}, \cite{vodopyanov_carnot_manifolds}]
\label{thm_gromov_blow_up}
If $U\subset M$ is an open subset with compact closure in $M$, then there exists $\rho_0>0$, and for $x\in U$ a family of smooth embeddings
$$
\{\de_\rho^x:\R^N\supset B(0,r_\rho)\ra M\}_{\rho<\rho_0}
$$
 and smooth vector fields $\{\hat X_j^{x,\rho}\}_{1\leq j\leq N}$ on $B(0,r_\rho)$, for $\rho\in [0,\rho_0)$, where:
\ben
\item $r_\rho\ra\infty$ as $\rho\ra 0$.
\item The family $\{\de_\rho^x\}$ depends continuously on $x$ and $\rho$ in the $C^\infty_{\loc}$-topology.
\item $\de_\rho^x(0)=x$
\item $(\de_\rho^x)^*X_j^\rho=\hat X_j^{x,\rho}$.
\item $\hat X_j^{x,\rho}\ra\hat X_j^{x,0}$  in $C^k_{\loc}$   as $\rho\ra 0$, uniformly on compact subsets in $x$, for every $k$.
\item The vector fields $\{\hat X_j^{x,0}\}$ span a graded Lie algebra of vector fields on $\R^N$, with grading given by degree
$$
\oplus_j\Span\{\hat X_i^{x,0}\mid \deg i = j\}\,;
$$
moreover  there is a graded isomorphism from $\fg_x$ to $\Span\{\hat X_j^{x,0}\}$ which sends $X_j(x)\mapsto \hat X_j^{x,0}$ for $\deg j=1$.
\item The vector fields $\{\hat X_j^{x,0}\}$ are given by the Baker-Campbell-Hausdorff formula (and in particular are polynomial vector fields); equivalently, the exponential map associated with  $\{\hat X_j^{x,0}\}$ is the identity map $\R^N\ra \R^N$.
\item If $\hat g^x_\rho$ is the Riemannian metric on $B(0,r_\rho)$ for which $\{\hat X^{x,\rho}_j\}$ is orthonormal, and $\hat d_{CC}^{x,\rho}$ is the corresponding Carnot-Caratheodory metric, then $\hat d_{CC}^{x,\rho}\ra \hat d_{CC}^{x,0}$  as $\rho\ra 0$ uniformly on compact subsets, uniformly in $x$.
\een
\end{theorem}

The weight of a subset $I\subset \{1,\ldots,N\}$ is defined to be $\wt I:=-\sum_{i\in I}\deg i$.  A differential form $\al=\sum_{I}a_I\th_I$ has weight $w$ (weight $\leq w$) if  $a_I=0$ unless $\wt(I)=w$ ($\wt I\leq w$); here $\th_I:=\La_{i\in I}\th_i$.

Next we define the center of mass for probability measures whose support has small diameter, imitating the construction for Carnot groups.

Let $K\subset M$ be a compact subset, and let $r_K>0$ be such that for every $x\in K$, the $\nabla$-exponential map $\exp_x$ is defined on $B_g(0,r_K)\subset T_xM$ and is a diffeomorphism onto its image, with inverse $\log_x$.  If $\nu$ is a compactly supported probability measure on $M$, and for some $x\in K$ we have $\spt(\nu)\subset \exp_x(B_g(0,r_K))$, then $(\log_x)_*\nu$ is a well-defined probability measure on $T_xM$, and we define $C_\nu(x)$ to be its center of mass, i.e.
$$
C_\nu(x):=\int_M\log_x\;d\nu=\int_{T_xM}y\;d((\log_x)_*\nu)(y)\,.
$$ 
We have:
\begin{lemma}
\label{lem_com_equiregular}
For every $\eps>0$ there  exists $r_K'=r_K'(\eps,K)\in (0,r_K)$ such that if $\spt\nu\subset K$, $d(\spt\nu,M\setminus K)>\eps$, and $\diam_g(\spt\nu)<r_K'$ then the set 
$$
\{x\in K\mid \spt\nu\subset B(x,10r_K')\}\,,
$$
contains a unique element $\com(\nu)$ such that $C_\nu(\com(\nu))=0$.
\end{lemma}
By working  locally in charts and rescaling, or by using Theorem~\ref{thm_gromov_blow_up}, Lemma~\ref{lem_com_equiregular} reduces to a perturbation of the Carnot group case; then the argument from Lemma~\ref{lem_carnot_center_of_mass}, which is robust under perturbation, may be applied.

Now let $(M',V_1',g',\{X_1',\ldots,X_{N'}'\})$ be a second subriemannian manifold; we will denote the associated data and objects with primes.  Let $f:M\supset U\ra U'\subset M'$ be a $W^{1,p}_{\loc}$-mapping  between open subsets, where $p$ is strictly larger than the homogeneous dimension of $(M,V_1)$.  Here $W^{1,p}_{\loc}$-mappings are defined as in the Carnot group case, except that instead of using horizontal left invariant vector fields to define the Sobolev space $W^{1,p}_{\loc}(M)$, one uses linear combinations of the vector fields $\{X_i\mid \deg i=1\}$.

Next we define a mollification procedure.   Choose a smooth, rotationally invariant probability measure $\si$ on $\R^n$, with $\spt\si\subset B(0,1)$, and for every $x\in M$, $\rho\in (0,\infty)$ we let $\si_{x,\rho}$ be the pushforward of $\si$ under the linear isomorphism $\R^n\ra T_xM$ which sends (the standard basis vector) $e_j$ to $X_j^\rho(x)=\rho^{\deg j}X_j(x)$.  If $K_0\subset U$ is a compact subset, then taking  $K_1$ compact with $K_0\subset \Int K_1\subset U$, and $K_1':=f(K_1)\subset U'$,    there is a $\rho_0>0$ such that if $\rho<\rho_0$ then applying Lemma~\ref{lem_com_equiregular} with $K:=K_1'$ and $\eps=\eps(f,\rho_0)$,   the measure $f_*\si_{x,\rho}$ will have a well-defined center of mass  for every $x\in K_0$; we define $f_\rho(x):=\com(f_*\si_{x,\rho})$.

We define the Pansu pullback of a differential form $\om\in \Om^*(U')$ by $f_P^*\om(x):=(D_Pf)(x)^*\om(f(x))$, where 
$$
D_Pf(x):T_xM\simeq\fg_x\ra \fg'_{f(x)}\simeq T_{f(x)}M'
$$ 
is the Pansu differential, which for a.e. $x\in U$ is a graded homomorphism  between  graded Lie algebras \cite{vodopyanov_carnot_manifolds}.

With these definitions, we have:

\begin{theorem}[Approximation Theorem for equiregular manifolds with an adapted frame]
\label{thm_equiregular_approximation_theorem}
Let $f:M\supset U\ra U'\subset M'$ be  a $W^{1,p}_{\loc}$-mapping for some $p>\nu$.
Suppose $\om\in \Om^k(U')$, $\eta\in\Om^{N-k}_c(U) $ are forms with continuous coefficients, such that $\wt(\om)+\wt(\eta)\leq -\nu$.  Then 
\begin{equation}  \label{eqn_main_approximation_alternative_equiregular}
  f_\rho^*\om\wedge \eta \stackrel{L^1_{\loc}}{\lra} f_P^*\om\wedge \eta
\end{equation}
where $f_\rho$ is the mollified map defined above.
Since $\eta$ has compact support, it follows that 

\begin{equation}  \label{eqn_main_approximation_equiregular}
\int_Uf_P^*\om\wedge \eta=\lim_{\rho\ra 0}\int_Uf_\rho^*\om\wedge \eta\,.
\end{equation}
\end{theorem}

We only give a brief indication of the proof  because it is very similar to the Carnot group case.  To prove Theorem~\ref{thm_equiregular_approximation_theorem}, we use the Dominated Convergence Theorem, following the proof of Theorem~\ref{th:main_approximation}.  

To verify pointwise convergence, one uses  the mappings
$$
\hat f^{x,\rho}:=(\de_\rho^{'f(x)})^{-1}\circ f\circ \de_\rho^x\,,
$$
where $\de_\rho^x$, $\de_\rho^{'f(x)}$ are obtained by applying Theorem~\ref{thm_gromov_blow_up} to $M$ and $M'$, respectively.
By \cite{vodopyanov_carnot_manifolds} if $x\in M$ is a point of differentiability of $f$, then  $\hat f^{x,\rho}\ra D_Pf(x)$ uniformly on compact sets, where $D_Pf(x):\R^N\ra\R^{N'}$; combining this with Theorem~\ref{thm_gromov_blow_up},  one obtains pointwise convergence as in the Carnot group case.  

Letting $U$ be an open set with compact closure containing $\spt\eta$, we use the same integrable majorant as in the Carnot group case (see \eqref{eqn_maximal_function}).  The proofs of integrability and majorization  follow along the same lines, using Theorem~\ref{thm_gromov_blow_up} to get the estimate  \eqref{eqn_maximal_function_controls_oscillation} for small $\rho$, by a slightly modified version of Lemma~\ref{lem_sobolev_embedding}.

Using this local approximation theorem we can easily extend the pullback theorem to equiregular manifolds.
In fact, for the validity of the pullback theorem we do not need global adapted frames.  We say that a tuple $(M,V_1,g)$ is
an  equiregular subriemannian manifold if
\bit
\item $M$ is a smooth $N$-manifold.
\item $V_1\subset TM$ is an equiregular subbundle, so we have a filtration of $TM$ by smooth subbundles
$$
\{0\}=V_0\subsetneq V_1\subsetneq\ldots\subsetneq V_s=TM
$$
where $V_j$ is spanned by brackets of smooth sections of $V_1$ and $V_{j-1}$ for $j\geq 2$, and $s \geq 1$.
\item  $g$ is a smooth Riemannian metric.
\eit
As before the homogeneous dimension of $M$ is defined by $\nu = \sum_{j=1}^sj\dim(V_j/V_{j-1})$.
For each $x \in M$ there exists an open neighbourhood $U_x$ such that $(U_x, V_1|_{U_x}, g|_{U_x})$ in addition admits an
adapted frame as above.

For a continuous differential form $\omega \in \Omega^k(M)$ we define a global  notion of weight as follows. 
For a one-form $\omega \in \Omega^1$ with $\omega|_{V_1} \equiv 0$ and $\omega \not \equiv 0$ we define 
\begin{equation} \wt(\omega) := -\max \{ j : \omega|_{V_j} \equiv 0  \} - 1.
\end{equation}
If $\omega|_{V_1} \not \equiv 0$ we set $\wt(\omega) = -1$.
We then define
\begin{equation}
\Omega^{k, \le w}(M) := \Span \{ \omega_1 \wedge \ldots \wedge \omega_k :  \om_j \in \Omega^1(M), \, \sum_{j=1}^k \wt(\omega_j) \le w\}
\end{equation}
and for $\om \not \equiv 0$ we set
\begin{equation}
\wt(\om) =  \wt(\omega, M) := \min \{ w: \om \in \Omega^{k,\le w}(M) \}.
\end{equation}
It is easy to see that if $U \subset M$ admits an adapted frame then 
$\wt(\omega, U)$ agrees with the definition of weight given above which was based on the adapted frame.
In particular the weight is independent of the choice of frame.

\begin{theorem}[Pullback theorem for equiregular manifolds]  \label{th:pull_back_equiregular} 
Let $M$ and $M'$ be equiregular manifolds, where $M$ has dimension $N$ and homogeneous dimension $\nu$, and suppose $f:M\supset U\ra U'\subset M'$ is a $W^{1,p}_{\loc}$  mapping between open sets for some $p>\nu$.
  Suppose  $\om\in \Om^k(U')$ is a continuous  form with continuous distributional exterior derivative $d\om$, suppose $\eta\in \Om^{N-k-1}_c(U)$ is smooth, and 
  \begin{equation}  \label{eq:weight_pullback_equiregular}
\wt(\om)+\wt(d\eta)\leq -\nu\,,\quad \wt(d\om)+\wt(\eta)\leq -\nu\,.
\end{equation}
Then  
\begin{equation}  \label{eq:pullback_theorem_equiregular}
\int_U(f_P^*d\om)\we\eta+(-1)^k\int_Uf_P^*\om\we d\eta=0\,.
\end{equation}
\end{theorem}
Here we use the convention that an  inequality in  \eqref{eq:weight_pullback_equiregular} holds if at least one of the forms involved vanishes.

\begin{proof} Since the proof is analogous to the proof for Carnot groups we just provide a sketch of the argument. 
The main point is to localize. By a partition of unity it suffices to show that for each $x \in U$ there exists an open neighborhood $U_x$ such that 
\eqref{eq:pullback_theorem} holds for all smooth $\eta\in \Om^{N-k-1}_c(U_x)$ which satisfy  \eqref{eq:weight_pullback_equiregular}.
Since $p > \nu$ the map $f$ is continuous. Thus we may choose $U_x$ so small that  there exist adapted frames in $U_x$ and  $U'_{f(x)}  \supset f(U_x)$.
Let $\tilde U \subset U_x$ be open and compactly contained in $U_x$ such that $\eta$ vanishes outside $\tilde U$. 
For sufficiently small $\rho > 0$ we have $f_\rho(\tilde U) \subset U'_{f(x)}$ and thus $d f_\rho^* \omega = f_\rho^* d\omega$ in the sense of distributions. The identity  
\eqref{eq:pullback_theorem_equiregular} now follows from 
Theorem~\ref{thm_equiregular_approximation_theorem} by taking $\rho \to 0$.
\end{proof}

\bibliography{product_quotient}
\bibliographystyle{amsalpha}

\end{document}